\documentclass[11pt,a4paper,reqno]{amsart}

\usepackage[a4paper,lmargin=2cm,rmargin=2cm,tmargin=4cm,bmargin=4cm]{geometry}

\usepackage[english]{babel}

\usepackage[centertags]{amsmath}
\usepackage{amsfonts}
\usepackage{amssymb}
\usepackage{amsthm}
\usepackage{dsfont}

\usepackage[dvipsnames]{xcolor}

\usepackage[normalem]{ulem}
\usepackage{bbm}

\usepackage{hyperref}
	\hypersetup{breaklinks=true,colorlinks=true,
linkcolor=MidnightBlue,citecolor=MidnightBlue,
urlcolor=MidnightBlue}
\usepackage{enumerate}

\makeatletter
\g@addto@macro\bfseries{\boldmath} 
\makeatother

\usepackage{tikz-cd,graphicx}
\usepackage{tikz}
\usepackage{subcaption}
\usetikzlibrary{shapes.geometric, arrows}
\usetikzlibrary{positioning,arrows,calc}
\tikzstyle{startstop} = [rectangle, rounded corners, minimum width=2cm, minimum height=1cm,text centered,text width=2cm, draw=black]
\tikzstyle{io} = [rectangle, rounded corners, minimum width=2cm, minimum height=1cm,text centered,text width=2cm, draw=black]

\tikzstyle{arrow} = [thick,->,>=stealth]
\tikzstyle{arrow2} = [dashed,->,>=stealth]

\usepackage{dot2texi}

\numberwithin{equation}{section}
\usepackage{enumerate}
\usepackage{mathtools}

\usepackage[shortlabels]{enumitem}

\newcommand{\abs}[1]{\left\lvert#1\right\rvert}
\newcommand{\norm}[1]{\left\lVert#1\right\rVert}

\newcommand{\N}{\mathbb{N}}

\newcommand{\K}{\mathbb{K}}

\newcommand{\eps}{\varepsilon}

\newcommand{\scal}[1]{\left\langle#1\right\rangle}

\newcommand{\nnorm}[1]{{\left\vert\kern-0.25ex\left\vert\kern-0.25ex\left\vert #1%
    \right\vert\kern-0.25ex\right\vert\kern-0.25ex\right\vert}}

\DeclareMathOperator{\dent}{dent}
\DeclareMathOperator{\Id}{Id}
\DeclareMathOperator{\diam}{diam}

\DeclareMathOperator{\conv}{conv}

\DeclareMathOperator{\cconv}{\overline{conv}}

\DeclareMathOperator{\real}{Re}

\DeclareMathOperator{\spn}{span}
\DeclareMathOperator{\cspn}{\overline{span}}
\DeclareMathOperator{\ext}{ext}
\DeclareMathOperator{\strexp}{str-exp}

\DeclareMathOperator{\PC}{PC}

\DeclareMathOperator{\cof}{cof}
\DeclareMathOperator{\Spear}{Spear}

\DeclareMathOperator{\re}{Re\,}
\newcommand{\T}{\mathbb{T}}
\newcommand{\R}{\mathbb{R}}
\renewcommand{\geq}{\geqslant}
\renewcommand{\leq}{\leqslant}

\theoremstyle{plain}
 \newtheorem{theorem}{Theorem}[section]
 \newtheorem{lemma}[theorem]{Lemma}
 \newtheorem{proposition}[theorem]{Proposition}
 \newtheorem{corollary}[theorem]{Corollary}

\theoremstyle{definition}
 \newtheorem{definition}[theorem]{Definition}
 \newtheorem*{definition*}{Definition}
 \newtheorem{remark}[theorem]{Remark}
 \newtheorem{example}[theorem]{Example}
 \newtheorem{question}[theorem]{Question}

\newcommand{\modaus}[2]{
\overline{\rho}_{#1}\left( #2 \right)
}

\begin{document}

\title{The super Alternative Daugavet property for Banach spaces}

\author[Langemets]{Johann Langemets}
\address[Langemets]{Institute of Mathematics and Statistics, University of Tartu, Narva mnt 18, 51009 Tartu, Estonia}
\email{johann.langemets@ut.ee}
\urladdr{\url{https://johannlangemets.wordpress.com/}}
\urladdr{
\href{https://orcid.org/0000-0001-9649-7282}{ORCID: \texttt{0000-0001-9649-7282} } }

\author[L\~{o}o]{Marcus L\~{o}o}
\address[L\~{o}o]{Institute of Mathematics and Statistics, University of Tartu, Narva mnt 18, 51009 Tartu, Estonia}
\email{marcus.loo@ut.ee}
\urladdr{
\href{https://orcid.org/0009-0003-1306-5639}{ORCID: \texttt{0009-0003-1306-5639} } }

\author[Mart\'in]{Miguel Mart\'in}
\address[Mart\'in]{Department of Mathematical Analysis and Institute of Mathematics (IMAG), University of Granada, E-18071 Granada, Spain}
\email{mmartins@ugr.es}
\urladdr{\url{https://www.ugr.es/local/mmartins/}}
\urladdr{
\href{http://orcid.org/0000-0003-4502-798X}{ORCID: \texttt{0000-0003-4502-798X} } }

\author[Perreau]{Yo\"el Perreau}
\address[Perreau]{Institute of Mathematics and Statistics, University of Tartu, Narva mnt 18, 51009 Tartu, Estonia}
\email{yoel.perreau@ut.ee}
\urladdr{
\href{https://orcid.org/0000-0002-2609-5509}{ORCID: \texttt{0000-0002-2609-5509} } }

\author[Rueda Zoca]{Abraham Rueda Zoca}
\address[Rueda Zoca]{Department of Mathematical Analysis and Institute of Mathematics (IMAG), University of Granada, E-18071 Granada, Spain}
\email{abrahamrueda@ugr.es}
\urladdr{\url{https://arzenglish.wordpress.com}}
\urladdr{
\href{https://orcid.org/0000-0003-0718-1353}{ORCID: \texttt{0000-0003-0718-1353} } }

\subjclass[2020]{Primary 46B20; Secondary 46B04, 46B22, 46B25}

\keywords{Radon--Nikod\'ym property, Asplund spaces, Daugavet property, Alternative Daugavet property} 

\date{\today}

\begin{abstract}
 We introduce the super alternative Daugavet property (super ADP) which lies strictly between the Daugavet property and the Alternative Daugavet property as follows. A Banach space $X$ has the super ADP if for every element $x$ in the unit sphere and for every relatively weakly open subset $W$ of the unit ball intersecting the unit sphere, one can find an element $y\in W$ and a modulus one scalar $\theta$ such that $\|x+\theta y\|$ is almost two. It is known that spaces with the Daugavet property satisfy this condition, and that this condition implies the Alternative Daugavet property. We first provide examples of super ADP spaces which fail the Daugavet property. We show that the norm of a super ADP space is rough, hence the space cannot be Asplund, and we also prove that the space fails the point of continuity property (particularly, the Radon--Nikod\'ym property). In particular, we get examples of spaces with the Alternative Daugavet property that fail the super ADP. For a better understanding of the differences between the super ADP, the Daugavet property, and the Alternative Daugavet property, we will also consider the localizations of these three properties and prove that they behave rather differently. As a consequence, we provide characterizations of the super ADP for spaces of vector-valued continuous functions and of vector-valued integrable functions.
\end{abstract}

\maketitle

{\parskip=0pt
\tableofcontents } 

\section{Introduction}
Given a Banach space $X$ over $\mathbb{K}$ ($\mathbb{K}=\mathbb{C}$ or $\mathbb{K}=\mathbb{R}$), we denote its dual as $X^*$, the unit ball and the unit sphere of $X$ as $B_X$ and $S_X$, respectively.
We say that $X$ satisfies the \emph{Daugavet property} (\emph{DP} for short), if the equation
\[
\norm{\Id + T} = 1+\norm{T}
\]
holds for every bounded rank-one operator $T\colon X \to X$, where $\Id$ is the identity operator of $X$ (see \cite{kssw} for background). A related property to the former is the \emph{Alternative Daugavet property} (\emph{ADP} for short), introduced in \cite{MT04}. The Banach space $X$ satisfies the ADP, if the norm equation
\[
\max_{\theta\in\mathbb{T}}\norm{\Id+\theta T}= 1+\norm{T}
\]
holds for every bounded rank-one operator $T\colon X\to X$. Here $\mathbb{T}$ denotes the set of all modulus one scalars of $\mathbb{K}$. It is obvious, by the definitions, that the DP implies the ADP. However, the converse fails - for example all the spaces $C(K)$, $L_1(\mu)$, and $L_1(\mu)$ preduals, have the ADP (even finite-dimensional ones) but, in general, not the DP, for which the perfection of $K$ or the non-atomicity of $\mu$ is needed. In particular, the ADP is compatible with the Radon--Nikod\'ym property (RNP) and Asplundness. However, spaces with the DP fail these properties. Also, a Banach space with the DP cannot be embedded into a Banach space with unconditional basis. 

Both the DP and the ADP admit a geometric characterization via slices of the unit ball. Recall that a \emph{slice} of a bounded subset $C$ of $X$ is the non-empty intersection of $C$ with an open half-space, that is, a set of the form
\[
S(C,x^*,\alpha) := \left\{x\in C\colon \real x^*(x) > \sup \real x^*(C)-\alpha\right\},
\]
where $x^*\in X^*\setminus \{0\}$ and $\alpha>0$. 

\begin{proposition}[{\cite[Lemma~2.2]{kssw}, \cite[Proposition~2.1]{MT04}}]\label{prop: geometric version of DP and ADP}\parskip=0pt
    Let $X$ be a Banach space.
    \begin{itemize}
        \item $X$ has the DP if and only if $
    \sup\limits_{y\in S}\|x+y\|=2$ for every $x\in S_X$ and every slice $S$ of $B_X$.
    \item $X$ has the ADP if and only if $
    \sup\limits_{y\in S}\max\limits_{\theta\in\T}\|x+\theta y\|=2$ for every $x\in S_X$ and every slice $S$ of $B_X$.
    \end{itemize}   
\end{proposition}

Shvydkoy observed in \cite[Lemma~3]{SHVYDKOY00} that the DP can also be characterized in terms of relatively weakly open subsets of the unit ball: $X$ has the DP if and only if for every $x\in S_X$ and every non-empty relatively weakly open subset $W$ of $B_X$, we have
$\sup\limits_{y\in W}\|x+y\|=2$. However, when one replaces slices with relatively weakly open subsets in the definition of the ADP, we arrive to a different property of Banach spaces.

\begin{definition}\label{def: SAD}
    We say that a Banach space $X$ has the \emph{super alternative Daugavet property} (\textit{super ADP} for short) if for every $x\in S_X$ and every non-empty relatively weakly open subset $W$ of $B_X$ intersecting $S_X$, we have 
    \begin{equation*}
        \sup_{y\in W}\max_{\theta\in\T}\norm{x+\theta y}=2. 
    \end{equation*}
\end{definition}
Let us remark that for infinite-dimensional spaces $X$, the requirement that the relatively weakly open subsets interesect the unit sphere is redundant. However, there is an easy example of finite-dimensional space with the super ADP: the one dimensional space, which is actually the only example, see Proposition~\ref{prop: K has the super ADP}.

From Shvydkoy's observation and the geometric characterization of the ADP, we immediately get the following chain of implications:
\[
\text{DP}\implies\text{super ADP}\implies\text{ADP.}
\]
We will see in Section~\ref{section: SAD spaces} that none of the above implications reverses.

Let us also comment that Shvydkoy's result actually works with convex combination of slices and that there is a characterization of the DP in these terms. This naturally opens the possibility of defining another property (say ``ccs ADP'') with the same spirit behind Definition~\ref{def: SAD}. However, this would actually lead to a characterization of the Daugavet property, see Corollary~\ref{corollary: ccs-AD is DP}. Therefore, the super ADP seems to be the only generalization that may produce a new property between the DP and the ADP.

Our main goal in this paper is to initiate the study of the super ADP. After a section on Notation and Preliminaries (Section~\ref{section:preliminaries}), we provide in Section~\ref{section: SAD spaces} the main properties and examples of Banach spaces with the super ADP. We begin by proving that the only Banach space satisfying the Kadec property (in particular, being finite-dimensional) with the super ADP is the one-dimensional one (Proposition~\ref{prop: K has the super ADP}). Hence, the $n$-dimensional $\ell_1$ and $\ell_\infty$ spaces (for $n$ greater than one) and $\ell_1$ are examples of spaces with the ADP but lacking the super ADP. Since there is no finite-dimensional space with the DP, this result will also distinguish super ADP spaces from the DP. Furthermore, we can also separate the above two properties in the infinite-dimensional setting: every Banach space with the DP can be equivalently renormed so that it has the super ADP, but lacks the DP (Theorem~\ref{theorem: renorming having SAD but not the DP}). We next study the isomorphic structure of the super ADP spaces, which looks more similar to the structure of DP spaces than to that of ADP ones: spaces having the super ADP will fail the convex point of continuity property (CPCP for short) (Theorem~\ref{theorem: super ADP fails CPCP}), hence fail the RNP, and their norms are rough (Theorem~\ref{theorem: super ADP is not Asplund}), hence they cannot be Asplund spaces. To get the result for the CPCP we provide a separable determination of the super ADP (Corollary~\ref{corollary:separbledetermination}).

Our second goal of this paper is to delve deeper into the differences between the ADP, the super ADP, and the Daugavet property by considering their corresponding ``localizations''. The investigation of pointwise versions of the Daugavet property was started in \cite{AHLP}, and  stronger versions were introduced in \cite{MPRZ}. Let us present the main definitions here. Let $X$ be a Banach space. A point $x\in S_X$ is said to be a \emph{Daugavet point} \cite{AHLP}, if for every slice $S$ of $B_X$, we have  $\sup_{y\in S}\|x+y\|=2$. Hence, a Banach space $X$ satisfies the DP if and only if every point of $S_X$ is a Daugavet point. Let us note here that Daugavet points are much more versatile than the global property: there exists a Banach space with RNP and a Daugavet point \cite{HLPV23,VeeorgStudia} and there exists a Banach space with a one-unconditional basis and a ``large'' subset of Daugavet points \cite{ALMT}. Stronger variants of Daugavet points were first introduced in \cite{MPRZ}. A point $x\in S_X$ is said to be a \emph{super Daugavet point}, if for every non-empty relatively weakly open subset $W$ of $B_X$, we have  $\sup_{y\in W}\|x+y\|=2$. By Shvydkoy's result \cite[Lemma~3]{SHVYDKOY00}, one has that $X$ satisfies the DP if and only if every point of $S_X$ is a super Daugavet point.  Further, a point  $x\in S_X$ is said to be a \emph{ccs Daugavet point}, if for every convex combination of slices $C$ of $B_X$, we have  $\sup_{y\in C}\|x+y\|=2$. Also by the Shvydkoy's result mentioned above, $X$ has the DP if and only if every element of $S_X$ is a ccs Daugavet point. Even though the three localizations provide the same global property, there are spaces were Daugavet points, super Daugavet points, and ccs Daugavet points differ from one another \cite{MPRZ}. 

Motivated by the above definitions, we will consider here similar localizations for the ADP and the super ADP.

\begin{definition}\parskip=0pt
    Let $X$ be a Banach space and $x\in S_X$. We say that $x$ is \begin{enumerate}
     \item an \emph{AD point} if for every slice $S$ of $B_X$, we have 
     $
     \sup\limits_{y\in S}\max\limits_{\theta\in\T}\norm{x+\theta y}=2;$
\item a \emph{super AD point} if for every non-empty relatively weakly open subset $W$ of $B_X$ intersecting $S_X$, we have 
  $     \sup\limits_{y\in W} \max\limits_{\theta\in\T}\norm{x+\theta y}=2.$
    \end{enumerate}
\end{definition}

We do not formally introduce the ``ccs AD points'' as they coincide with ccs Daugavet points, see Proposition~\ref{proposition: ccs AD is the same as ccs Daugavet}.

From the definitions above, one immediately has that a Banach space has the ADP (respectively, the super ADP) if and only if every point in the unit sphere is an AD point (respectively, super AD point). We devote Section~\ref{section: Alternative Daugavet points and related notions}
to study AD points and super AD points. In Subsection~\ref{subsection:examplesandcomparision} we start by relating these new pointwise notions to some existing diametral notions and to the concept of spear vectors of \cite{Ardalani2014,KMMP}.  (Figure~\ref{fig:relations between notions} contains a diagram of the relations between the diametral notions.) This allows us to present a description of the super AD points in $\ell_\infty^m$ for $m\in \N$ and in $\ell_1(\Gamma)$ spaces (Example~\ref{example:ellinftym-ell1Gamma}), and to show that $c_0$ has no super AD points (Example~\ref{example: c_0 AD but no SAD points}). Next, we improve some results on Daugavet and super Daugavet points, showing that they are AD points and super AD points, respectively, ``in every direction'' (Propositions \ref{prop:superDpoint-is-superADPineverydirection} and \ref{prop: Daugavet point makes all rotations far for uniform element}). We next provide some relations between AD points and super AD points with denting points and points of continuity, respectively, which allow to give characterizations of the two notions for RNP and CPCP spaces, see Corollary~\ref{corollary: characterizing AD and super AD points RNP and CPCP}. The subsection ends showing that ``ccs AD points'' are actually ccs Daugavet points, Proposition~\ref{proposition: ccs AD is the same as ccs Daugavet}. Subsection~\ref{subsec:implications-of-points-in-geometry} is devoted to the implications of AD points and super AD points for the geometry of the underlying space. We show that no AD point can be a LUR point (so, it cannot be a point of uniform convexity) unless the dimension is one, see Proposition~\ref{prop: LUR points no AD points}. We also show that for Banach spaces with the CPCP, no point of G\^{a}teaux differentiability of the norm can be a super AD point (Proposition~\ref{prop:PC-wdense-GdifnosuperAD}). Finally, asymptotically smooth points of the unit sphere of an infinite-dimensional Banach space cannot be super AD points (Proposition~\ref{prop: frechet diff points is not super AD point}), generalizing the fact that $c_0$ has no super AD point. As a consequence, infinite $\ell_p$- ($1<p<\infty$) and $c_0$-sums of finite-dimensional Banach spaces contain no super AD points.

Our goal in Section~\ref{section: AD and super AD points in classical Banach spaces} is to study the super AD points and the super ADP in spaces of (vector-valued) continuous functions and of (vector-valued) integrable functions. We start in Subsection~\ref{subsec: SAD in absolute sums} studying the super AD points of $\ell_1$- and $\ell_\infty$-sums of Banach spaces, getting results which sometimes differ from the known ones for the super Daugavet property. Next, we apply these results (and the known results about the super Daugavet points) to get characterizations of  super AD points in spaces of vector-valued integrable functions (Proposition~\ref{prop:Daugavet points in L_1(mu,X)}) and in spaces of vector-valued continuous functions (Proposition~\ref{prop: super AD in C(K,X)}). All the previous results in this section allows us to present in Subsection~\ref{subsec:exmples of super AD spaces} characterizations of the super ADP for spaces of vector-valued integrable functions (Theorem~\ref{theorem:characterizing-superADPL1muX}) and for spaces of vector-valued continuous functions (Theorem~\ref{theorem:characterizing-superADPL1muX}).

\section{Notation and preliminaries}\label{section:preliminaries}
In this short section, we recall a few classical notions from Banach space geometry that we will be using throughout the text.
The notation of this paper is standard, following along the lines of \cite{FHPPMZ}.

Let $X$ be a Banach space and let $A$ be a non-empty closed convex bounded subset of $X$. We denote by $\spn(A)$, $\conv(A)$, $\cspn(A)$, and $\cconv(A)$ the linear span and convex hull of $A$, as well as their respective closures. The following notions are well studied properties for points of $A$.
We say that a point $x$ in $A$ is 
\begingroup\parskip=0pt
\begin{enumerate}
    \item an \emph{extreme point} of $A$ (writing $a\in \ext(A)$) if $x$ does not belong to the interior of any segment of $A$;
    \item a \emph{point of continuity} of $A$ (writing $a\in \PC(A)$) if the identity mapping $\Id\colon (A,w)\to(A,\norm{\cdot})$ is continuous at $x$ (that is, if $x$ is contained in relatively weakly open subset of $A$ of arbitrarily small diameter);
    \item a \emph{denting point} of $A$ (writing $a\in \dent(A)$) if $x$ is contained in slices of $A$ of arbitrarily small diameter;
    \item a \emph{strongly exposed point} of $A$ (writing $a\in \strexp(A)$) if there exists $x^*\in X^*$ such that for all sequences $(x_n)\subset A$,  $\re x^*(x_n)\rightarrow \re x^*(x)$ if and only if $x_n\rightarrow x$ in norm.
\end{enumerate} 
\endgroup
 
Clearly, strongly exposed points are denting, and denting points are both extreme points and points of continuity. In fact, denting points are precisely those points which are simultaneously extreme points and points of continuity (see \cite[Exercise 3.146]{FHPPMZ}, for instance). For the case of $A=B_X$, a stronger version of strongly exposed point is the one of LUR point. An element $x\in B_X$ is a \emph{LUR point} if for every $\eps>0$ there is $\delta>0$ such that the implication
    \[
    \norm{\frac{x+y}{2}}>1-\delta \implies \norm{x-y}\leq \varepsilon
    \]
holds for every $y\in B_X$. If every element of $S_X$ is LUR, we say that the space $X$ is \emph{LUR}.

Recall that a Banach space $X$ has the \emph{Radon--Nikod\'ym property} (\textit{RNP} for short) if every closed convex bounded subset of $X$ has a denting point. Also, a Banach space $X$ has the \emph{point of continuity property} (\textit{PCP} for short), respectively, the \emph{convex point of continuity property} (\textit{CPCP} for short), if every closed bounded subset, respectively every closed convex bounded subset, of $A$ has a point of continuity. It follows that in a RNP space $X$, $B_X$ is the closed convex hull of the set of all denting points of $B_X$, and that in spaces $X$ with the CPCP, the set of all points of continuity of $B_X$ is weakly dense in $B_X$, see \cite{ggsm} for more information and background. A related isometric notion is the following: a Banach space has the \emph{Kadec property} if the identity map $\Id\colon (B_X,w)\to (B_X,\norm{\cdot})$ is continuous on the whole $S_X$ (in other worlds, if $S_X\subset \PC(B_X)$), see e.g.\ \cite[Section II.1]{DGZ}. Finite-dimensional Banach spaces trivially satisfy the Kadec property, and also uniformly convex spaces or, more generally, those Banach spaces for which $\dent(B_X)=S_X$, such as LUR spaces. Moreover, asymptotically uniformly convex spaces satisfy a uniform version of the Kadec property (see e.g.\ the remark following \cite[Proposition~2.6]{JLPS}, or the discussion at the end of Section~2 in \cite{ALMP}). Recall that these latter spaces include in particular all $\ell_p$-sums of finite-dimensional spaces, $1\leq p<\infty$.

\begingroup \parskip=0pt
Related to differentiability of norms are the following notions. We say that a point $x$ in a Banach space $X$ is 
\begin{enumerate}
        \item a point of \emph{G\^{a}teaux differentiability} of $X$ if there exists a unique functional $f\in S_{X^*}$ such that $f(x)=\norm{x}$;
        \item a point of \emph{Fr\'{e}chet differentiability} of $X$ if there exists a functional $f\in S_{X^*}$ such that for every sequence $(f_n)$ in $S_{X^*}$, $f_n(x)\to \norm{x}$ if and only if $f_n\to f$ in norm. 
    \end{enumerate}
Recall that a Banach space $X$ is \emph{Asplund} if every continuous and convex function $f$ from a non-empty open subset $U$ of $X$ into $\K$ is Fr\'{e}chet differentiable on a $G_\delta$ subset of $U$. Equivalently, $X$ is Asplund if and only if every separable subspace of $X$ has a separable dual if and only if $X^*$ has the RNP. 
\endgroup

From its various geometric characterizations, the Daugavet property admits several natural localizations to points of the unit sphere of Banach spaces, as we already mentioned in the introduction. Let us present two more notions, and refer to \cite{HLPV23} for more information and background. Let $X$ be a Banach space and $x\in S_X$. We say that $x$ is 
\begingroup \parskip=0pt
\begin{enumerate}
      \item a \emph{$\nabla$-point} if for every slice $S$ of $B_X$ not containing $x$, we have $\sup\nolimits_{y\in S} \norm{x-y}=2$.
      \item a \emph{$\Delta$-point} if for every slice $S$ of $B_X$ containing $x$, we have $\sup\nolimits_{y\in S} \norm{x-y}=2$.
    \end{enumerate}
Observe that a point is a Daugavet point if and only if it is $\nabla$ and $\Delta$ simultaneously. Also, recall that finite-dimensional spaces contain no Daugavet points (see \cite{AALMPPV2}).
\endgroup

Finally, we end the section with a brief recap about spear vectors. Let $X$ be a Banach space.  An element $x\in S_X$ is a \emph{spear vector} (or just \emph{spear}) \cite{Ardalani2014,KMMP} if $\max\limits_{\theta\in\T}\|x+\theta y\|=2$ for every $y\in S_X$; equivalently, if for every $x^*\in\ext(B_{X^*})$, we have $\abs{x^*(x)}=1$, see \cite{KMMP} for more information and background. It is immediate that every spear is an extreme point of $B_X$, while the opposite is not always true. Examples of spears include the elements of the unit vector basis of $\ell_1(\Gamma)$, the extreme points of the unit ball of $\ell_\infty^n$ and $\ell_\infty$, or the constant $1$ function in $C[0,1]$. Observe that, in the real case, two distinct spears are necessarily at distance $2$ from one another (because they have to assume a different value on some extreme point). In the complex case, we have that $\Spear(X)$ is nowhere dense in $S_X$ unless $X$ is one-dimensional \cite[Proposition~2.11]{KMMP}.  From this discussion, the following result follows.

\begin{proposition}[\cite{KMMP}]\label{prop:stricly_convex_spaces_dim>1_have_no_spears}
Let $X$ be a Banach space. If $\Spear(X)=S_X$, then $X$ is one-dimensional.
\end{proposition}

An easy convexity argument gives the following remark.

\begin{remark}\label{remark:convexityargument}
Let $X$ be a Banach space, let $x,y\in S_X$, and let $0<\eps<2$. If $\|x+y\|>2-\eps$, then $\|ax+by\|>a+b-\eps$ for every $a,b\in [0,1]$.
\end{remark}

If $X$ and $Y$ are Banach spaces over $\K$, we write $X\oplus_p Y$ to denote the $\ell_p$-sum for $1\leq p\leq \infty$. For a generalization of the above, given a family $\{E_\gamma\colon \gamma\in \Gamma\}$ of Banach spaces,  $\left[\bigoplus_{\gamma\in\Gamma}E_\gamma\right]_{\ell_p}$ represent the $\ell_p$-sum of the family for $1\leq p\leq \infty$ and  $\left[\bigoplus_{\gamma\in\Gamma}E_\gamma\right]_{c_0}$ is the $c_0$-sum of the family.

\section{Super ADP for Banach spaces}\label{section: SAD spaces}
Our goal here is to discuss the first examples and several properties of the super ADP.

Clearly, the Daugavet property implies the super ADP, and the super ADP implies the ADP. Our first result shows that the super ADP lies strictly between the ADP and the Daugavet property. 

\begin{proposition}\label{prop: K has the super ADP}
The one dimensional space $\K$ is the only finite-dimensional Banach space which has the super ADP. Moreover, $\K$ is the only super ADP space with the Kadec property.
\end{proposition}

\begin{proof}
Clearly, $\K$ has the super ADP. Also, let us observe that if a Banach space $X$ is super ADP, then for every $x\in S_X$ and $y\in \PC(B_X)\cap S_X$, we have that $\max_{\theta\in\T}\norm{x+\theta y}=2$. Indeed, for every $\eps>0$, consider a weak neighborhood $W$ of $y$ of diameter smaller than $\eps$.  Then, we can find $z\in W$ such that $\max_{\theta\in\T}\norm{x+\theta z}>2-\eps$. It follows that $\max_{\theta\in\T}\norm{x+\theta y}>2-2\eps$ for every $\eps>0$, getting the desired result. If $X$ has the Kadec property, this equality does hold for every $y\in S_X$, which means that every $x\in S_X$ is a spear of $X$. By Proposition~\ref{prop:stricly_convex_spaces_dim>1_have_no_spears}, it follows that $X=\K$. 
\end{proof}

Recall that $\ell_1(\Gamma)$ has the Kadec property, since it is asymptotically uniformly convex. It follows that these spaces fail the super ADP, if the set $\Gamma$ has more than one element. Moreover, it is known that these spaces have the ADP, providing infinite-dimensional examples separating these two properties.

\begin{example}\label{examples: ell1gamma not super ADP}
For every set $\Gamma$ with more than one element, the space $\ell_1(\Gamma)$ fails to have the super ADP.
\end{example}

A stronger result than the previous one will be provided in Theorem~\ref{theorem: super ADP fails CPCP}.

We will now show that even in infinite-dimensional Banach spaces, the super ADP is strictly weaker than the Daugavet property. 

\begin{theorem}\label{theorem: renorming having SAD but not the DP}
Let $X$ be a Banach space with the Daugavet property. Then, $X$ can be renormed to simultaneously satisfy the super ADP and fail the Daugavet property.
\end{theorem}

We provide here the following immediate lemma, which also has a local version, see Lemma~\ref{lemma:SAD_net_characterization}.

\begin{lemma}\label{lemma:global SAD_net_characterization}
Let $X$ be a Banach space. Then, $X$ has the super ADP if and only if given any two elements $x,y\in S_X$ and $\delta>0$, there is a net $(y_{\alpha})$ in $S_X$ which weakly converges to $y$ and satisfies that $\limsup_\alpha\max_{\theta\in\mathbb{T}}\|x+\theta y_{\alpha}\|\geq 2-\delta$. Moreover, it is enough to show this for a subset $C$ of elements $x$'s in the unit sphere such that $\T C$ is dense and for a subset $D$ of elements $y$'s in the unit sphere such that $\T D$ is dense.
\end{lemma}

\begin{proof}[Proof of Theorem~\ref{theorem: renorming having SAD but not the DP}]
Let $(X,\norm{\cdot})$ be a Banach space with the DP, and fix $x_0\in S_{(X,\norm{\cdot})}$ and $\varepsilon\in(0,1)$. We consider an equivalent norm $|||\cdot|||$ on $X$ whose unit ball is defined as
\begin{equation}\label{eq: definition of the new unit ball}
B_{(X,|||\cdot|||)} := \conv\Bigl(\T x_0\cup (1-\varepsilon)B_{(X,\norm{\cdot})}\Bigr).
\end{equation}
Observe that this set is closed, since $\T x_0$ is compact and $(1-\varepsilon)B_{(X,|||\cdot|||)}$ is closed. Let \begin{equation*}
    A:=\T x_0\cup (1-\varepsilon)B_{(X,\norm{\cdot})},
\end{equation*} and observe that, using \eqref{eq: definition of the new unit ball}, we can compute the norm of $f\in (X^*,|||\cdot|||)$ as follows:
\begin{equation}\label{eq: the new dual norm}
|||f||| = \sup_{x\in A}\abs{f(x)} = \max\left\{ \sup_{x\in(1-\varepsilon)B_{(X,\norm{\cdot})}}\abs{f(x)},\, \abs{f(x_0)}\right\} = \max\{(1-\varepsilon)\norm{f}, \abs{f(x_0)}\}.
\end{equation} We claim that the Banach space $(X,|||\cdot|||)$ has the super ADP, but not the DP.

Notice that $x_0$ is clearly a denting point in the unit ball $B_{(X,|||\cdot|||)}$ (see e.g.\ \cite[Lemma~2.1]{BLR16}). In particular, $(X, |||\cdot|||)$ fails the DP. To show that it has the super ADP, it suffices to find, given $x,y\in S_{(X,|||\cdot|||)}$ and $\delta>0$, a net $(y_{\alpha})$ which weakly converges to $y$ and satisfies $\limsup_\alpha\max_{\theta\in\mathbb{T}}|||x+\theta y_{\alpha}|||\geq 2-\delta$, see Lemma~\ref{lemma:global SAD_net_characterization}. Actually, by the moreover part of Lemma~\ref{lemma:global SAD_net_characterization}, we can assume that
\begin{align*} 
    x &= \lambda x_0 +(1-\lambda)(1-\varepsilon) u, \;\; \lambda\in(0,1), \; u\in B_{(X,||\cdot||)}; \\
    y &= \mu \omega x_0 + (1-\mu)(1-\varepsilon)v, \;\; \mu\in(0,1),\; v\in B_{(X,||\cdot||)},\; \omega\in\T.
\end{align*} 
Fix $\eta\in(0,\delta/2)$. Since space $(X,\norm{\cdot})$ has the DP, then using \cite[Lemma 2.8]{kssw} and \cite[Lemma 3]{SHVYDKOY00}, we can find a net $(v_{\alpha})_{\alpha\in I}$ weakly converging to $\omega^{-1} v$ and satisfying
\begin{equation}\label{eq: Daugavet ell_1 estimation}
\norm{k_1x_0+k_2u+k_3v_{\alpha}}\geq (1-\eta)\big(\norm{k_1x_0+k_2u} + \abs{k_3}\big),\;\;\;\;k_1,k_2,k_3\in\mathbb{K},\;\;\alpha\in I.
\end{equation} 
We now take $f\in (X^*,|||\cdot|||)$, $|||f|||=1$, such that $f(x)=1$. By the representation of $x$ above, this implies that $f(x_0)=1$ and $f((1-\varepsilon)u)=1$. For $\alpha\in I$, consider the linear functional $h_{\alpha}\colon \spn\{x_0,u,v_{\alpha}\} \to \mathbb{K}$ defined as
\begin{equation*}
k_1x_0+k_2u + k_3 v_{\alpha} \longmapsto k_1f(x_0)+k_2f(u) + k_3f(u)
\end{equation*}
for $k_1,k_2,k_3\in\mathbb{K}$. 
Using \eqref{eq: Daugavet ell_1 estimation}, we obtain a bound for $||h_{\alpha}||$: for  any $k_1x_0+k_2u+k_3v_{\alpha}$, we have
\begin{align*}
    \abs{h_{\alpha}(k_1x_0+k_2u+k_3v_{\alpha})} &= \abs{f(k_1x_0+k_2u) + k_3f(u)}
    \leq \abs{f(k_1x_0+k_2u)}+ \abs{k_3}\norm{f}\\
    &\leq \norm{f}\norm{k_1x_0+k_2u} + |k_3|\norm{f}\leq \norm{f}\big(\norm{k_1x_0+k_2u}+\abs{k_3}\big)\\
    &\leq\frac{\norm{f}}{1-\eta}\norm{k_1x_0+k_2u+k_3v_{\alpha}}.
\end{align*}
Consequently, $\norm{h_{\alpha}}\leq \norm{f}/(1-\eta)$. Moreover,  \eqref{eq: the new dual norm} yields 
\begin{equation*}
|||h_{\alpha}|||= \max\bigl\{(1-\varepsilon)\norm{h_{\alpha}}, |h(x_0)|\bigr\}\leq \max\Big\{\tfrac{1-\varepsilon}{1-\eta}\norm{f}, 1\Big\}.
\end{equation*}
Since $|||f|||=1$ and $\abs{f(x_0)}=1$, we have $\norm{f}\leq 1/(1-\varepsilon)$, so
\begin{equation*}
|||h_{\alpha}|||\leq \max\Big\{\tfrac{1}{1-\eta},1\Big\}=\frac{1}{1-\eta}.
\end{equation*}
For every $\alpha\in I$, we now extend $h_{\alpha}$, using the Hahn--Banach Theorem, to the space $(X,|||\cdot|||)$. Therefore, we have $h_{\alpha}\in(X^*,|||\cdot|||)$ with $|||h_{\alpha}|||<1/(1-\eta)$. Denote
$y_{\alpha} :=\mu\omega x_0 + (1-\mu)(1-\varepsilon)\omega v_{\alpha}$ for every $\alpha\in I$. Since $(v_{\alpha})_{\alpha\in I}$ converges weakly to $\omega^{-1} v$, the net $(y_{\alpha})_{\alpha\in I}$ converges weakly to $y$. We conclude the proof by showing that, for every $\alpha\in I$, we have 
\begin{equation*}
\max_{\theta\in\mathbb{T}}|||x+\theta y_{\alpha}|||\geq 
|||x+\omega^{-1} y_{\alpha}|||>2-\delta.
\end{equation*}
Indeed, for a fixed $\alpha\in I$, we have
\begin{align*}
 |||x+\omega^{-1}y_{\alpha}|||&\geq \abs{\frac{h_{\alpha}}{|||h_{\alpha}|||}\Big(x+\omega^{-1}y_{\alpha}\Big)}\\
 &\geq (1-\eta)\Big|h_{\alpha}\big(\lambda x_0+(1-\lambda)(1-\varepsilon)u+\omega^{-1}(\mu\omega x_0 + (1-\mu)(1-\varepsilon)\omega v_{\alpha})\big)\Big|\\
 &=(1-\eta)\Big|\lambda f(x_0)+(1-\lambda)(1-\varepsilon)f(u)+\mu f(x_0)+(1-\mu)(1-\varepsilon)f(u)\Big|\\
 &=(1-\eta)\Big|\lambda+(1-\lambda)+\mu +(1-\mu)\Big|=2(1-\eta)=2-2\eta>2-\delta.\qedhere
\end{align*}
\end{proof}

Having established that the super ADP lies strictly between the ADP and the Daugavet property, it is natural to ask what kind of constraints the super ADP imposes on the underlying Banach space. Our first result in this line is that separable Banach spaces with the CPCP cannot have the super ADP.

\begin{proposition}\label{prop:separable_CPCP_not_super_ADP}
Let $X$ be a Banach space of dimension greater than or equal to two satisfying that the set $\PC(B_X)$ is weakly dense in $B_X$ (in particular, if $X$ has the CPCP) and that the set of points of G\^{a}teaux differentiability of the norm of $X$ is not empty (in particular, if $X$ is separable). Then, $X$ does not satisfy the super ADP. 
\end{proposition}

\begin{proof}
Let $x$ be a point of G\^{a}teaux differentiability of the norm. Then there exists a unique $f\in S_{X^*}$ such that $f(x)=1$. Consider the relatively weakly open subset $W:=\{y\in B_X\colon \abs{f(y)}<1/2\}$ of $B_X$. Since $\dim(X)\geq 2$, $W$ must intersect $S_X$. By the hypothesis, $W$ must contain a point of continuity $y$ of $B_X$ that belongs to $S_X$. Let us now prove that $\rho:=\max_{\theta\in\T}\norm{x+\theta y}<2$. Indeed, if 
$\theta\in\T$ satisfies that $\norm{x+\theta y}=2$, we pick $g\in S_{X^*}$ such that $\re g(x+\theta y)=\norm{x+\theta y}=2$, so $\re g(x)=1$ and $g=f$ by uniqueness of the supporting functional. But then, $$\norm{x+\theta y}=|f(x+\theta y)|\leq 1 + |f(y)|<3/2$$
as $y\in W$. Now, we use that $y\in \PC(B_X)$ to find a weakly open subset $W_1$ of $B_X$, such that $y\in W_1$ (so $W_1$ intersects $S_X$) and with $\diam(W_1)<(2-\rho)/2$. 
If $z\in W_1$, we have that
\begin{align*}
\max_{\theta\in\T}\norm{x+\theta z} &\leq \max_{\theta\in\T}\norm{x+\theta y} + \|z-y\| \leq \rho +\diam(W_1) \leq 1+ \frac{\rho}{2}<2.
\end{align*}
Hence, $\sup_{z\in W_1}\max_{\theta\in\T}\norm{x+\theta z}<2$ and so $X$ fails the super ADP.
\end{proof}

In order to get a non-separable version of the above result, we will now prove that the super ADP is separably determined by a.i.\ ideals. Let $Y$ be a Banach space and $X$ be a subspace of $Y$. Recall that $X$ is said to be an \emph{almost isometric ideal} (\textit{a.i.\ ideal} for short) of $Y$ if for every $\eps>0$, and for every finite-dimensional subspace $E$ of $Y$, there exists a bounded linear operator $T\colon E\to X$ satisfying:
\begingroup \parskip=0pt \begin{enumerate}
    \item for every $e\in E\cap X$, $T(e)=e$;
    \item for every $e\in E$, $(1-\eps)\norm{e}\leq \norm{T(e)}\leq (1+\eps)\norm{e}$.
\end{enumerate} \endgroup
Also recall that a bounded linear operator $\varphi\colon X^*\to Y^*$ is called a \emph{Hahn--Banach extension operator} if $\norm{\varphi x^*}=\norm{x^*}$ for every $x^*\in X^*$ and $\varphi x^*(x)=x^*(x)$ for every $(x,x^*)\in X\times X^*$. The following result was proved in \cite{ALN14}.

\begin{lemma}[\mbox{\cite[Theorem~1.4]{ALN14}}]
    Let $Y$ be a Banach space and let $X$ be an a.i.\ ideal of $Y$. Then there exists a Hahn--Banach extension operator $\varphi\colon X^*\to Y^*$ such that for every $\eps>0$, for every finite-dimensional subspace $E$ of $Y$, and for every finite-dimensional subspace $F$ of $X^*$, there exists a bounded linear operator $T\colon E\to X$ satisfying: \begingroup \parskip=0pt \begin{enumerate}
    \item for every $e\in E\cap X$, $T(e)=e$;
    \item for every $e\in E$, $(1-\eps)\norm{e}\leq \norm{T(e)}\leq (1+\eps)\norm{e}$;
    \item for every $(e,f)\in E\times F$, $\varphi f(e)=f(Te)$.
\end{enumerate}
\endgroup
\end{lemma}

We may now get the stability of the super ADP by a.i.\ ideals.

\begin{proposition}
    Let $Y$ be Banach space with the super ADP and let $X$ be an a.i.\ ideal in $Y$. Then, $X$ also has the super ADP. 
\end{proposition}

\begin{proof}\parskip=0pt 
    Take $x\in S_X$ and $\eps>0$. Then let $x_0\in S_X$ and let $W$ be a neighborhood of $x_0$ in the relative weak topology of $B_X$. Without lost of generality, there exists $n\in\N$, $x_1^*,\dots, x_n^*$ in $X^*\setminus\{0\}$ and $\delta>0$ such that \begin{equation*}
        W:=\bigcap_{i=1}^n\{x\in B_X\colon \abs{x^*_i(x-x_0)}<\delta\}.
    \end{equation*} Let $\varphi\colon X^*\to Y^*$ be the Hahn--Banach extension operator given by the previous lemma, and consider the set 
    \begin{equation*}
        \widetilde{W}:=\bigcap_{i=1}^n\{y\in B_Y\colon \abs{\varphi x^*_i(y-x_0)}<\delta/2\}.
    \end{equation*} 
    Since $x$ is a super AD point in $Y$, there exists $y\in \widetilde{W}$ and $\theta\in\T$ such that $\norm{x+\theta y}>2-\eps$. Let $\eta>0$ to be chosen later. Applying the above with $E:=\spn\{x_0,x,y\}$ and $F:=\spn\{x_1^*,\ldots, x_n^*\}$, we get a bounded linear operator $T\colon E\to X$ satisfying: 
    \begin{enumerate}
        \item $Tx_0=x_0$ and $Tx=x$;
        \item for every $e\in E$ and for every $i\in\{1,\dots,n\}$, $\varphi \bigl(x_i^*(e)\bigr)=x_i^*(Te)$;
        \item for every $e\in E$, $(1-\eta)\norm{e}\leq \norm{Te}\leq (1+\eta)\norm{e}$.
    \end{enumerate} 
    For simplicity, let us now distinguish two cases. First, assume that $\norm{Ty}>1$ and let $z:=\frac{Ty}{\norm{Ty}}$. By assumption, we have that $\norm{z-Ty}=\norm{Ty}-1\leq \eta$. Thus, for every $i\in\{1,\dots,n\}$, we have \begin{equation*}
        \abs{x_i^*(z-x_0)}\leq \norm{z-Ty}\cdot\norm{x_i^*}+\abs{x_i^*(Ty-x_0)} = \norm{z-Ty}\cdot\norm{x_i^*}+\abs{\varphi x_i^*(y-x_0)} < \eta\cdot\norm{x_i^*}+\delta/2.
    \end{equation*} Furthermore, \begin{equation*}
        \norm{x+\theta z}\geq \norm{x+\theta Ty}-\norm{z-Ty}=\norm{Tx+\theta Ty}-\norm{z-Ty}\geq (1-\eta)(2-\eps)-\eta. 
    \end{equation*} So if $\eta$ was initially chosen so that $\eta\cdot\norm{x_i^*}<\delta/2$ for every $i$ and so that $(1-\eta)(2-\eps)-\eta>2-2\eps$, then we would get that $z\in W$ and $\norm{x+\theta z}>2-2\eps$. The conclusion follows. The case $\norm{Ty}\leq 1$ can be dealt with analogously, and with simpler computations, because we can then simply take $z=Ty$.
\end{proof}

The separable determination of the super ADP now immediately follows from \cite[Theorem~1.5]{Abrahamsen15} which assures that given a Banach space and a separable subspace, there is a separable a.i.\ of the space containing the subspace.

\begin{corollary}\label{corollary:separbledetermination}
Let $X$ be an infinite-dimensional Banach space with the super ADP and let $Z$ be a separable subspace of $X$. Then, there is a separable subspace $W$ of $X$ with the super ADP which contains $Z$. 
\end{corollary}

Since the CPCP clearly passes to subspaces, we immediately get the following.  

\begin{theorem}\label{theorem: super ADP fails CPCP}
Spaces with the CPCP (in particular, spaces with the PCP or with the RNP) of dimension greater than or equal to two, do not have the super ADP. 
\end{theorem}

Finally, we will show that spaces with the super ADP fail to be Asplund in a rather strong way. Recall that the norm of a Banach space $X$ is said to be \textit{$\rho$-rough} for $0<\rho\leq 2$ if  
$$
\limsup_{\|h\|\to 0} \frac{\|x+h\|+\|x-h\|-2\|x\|}{\|h\|}\geq \rho
$$ 
for all $x\in X$, see \cite{DGZ} for background. Observe that the roughness of the norm is the extreme opposite to the Fr\'{e}chet differentiability. It is known that the norm of $X$ is $\rho$-rough if and only if every weak$^*$ slice of $B_{X^*}$ has diameter greater than or equal to $\rho$ \cite[Proposition I.1.11]{DGZ}. Observe that spaces admitting a $\rho$-rough norm cannot be Asplund, as this norm is not Fr\'{e}chet differentiable at any point. In particular, the following shows that super ADP spaces fail to be Asplund. 

\begin{theorem}\label{theorem: super ADP is not Asplund}
Let $X$ be a Banach space of dimension greater than or equal to two with the super ADP. Then, for every weak$^*$ slice $S$ of $B_{X^*}$ and every $x^*\in S$, we have 
$\sup\nolimits_{y^*\in S}\norm{x^*-y^*}\geq 1$.
In particular, every slice of $B_{X^*}$ has radius, hence diameter, greater than or equal to one, so the norm of $X$ is $1$-rough. Therefore, $X$ is not an Asplund space. 
\end{theorem}

\begin{proof}
Fix $x\in S_X$ and $\delta>0$ and consider the weak$^*$ slice $S(B_{X^*},x,\delta)$. Let $x^*\in S(B_{X^*},x,\delta)$, pick $\eps\in(0,\delta)$, and let $W:=\{y\in B_X\colon \abs{x^*(y)}<\eps\}$ (observe that $W$ intersects $S_X$ since $\dim(X)>1$). As $X$ has the super ADP and $W$ is balanced (i.e. is such that $\T W=W$), there exists $y\in W$ such that $\norm{x+y}>2-\eps$. Let $y^*\in S_{X^*}$ be such that $\real y^*(x+y)>2-\eps$. On the one hand, we have $\real y^*(x)>1-\eps>1-\delta$, so $y^*\in S(B_{X^*},x,\delta)$. On the other hand, 
\[
\norm{x^*-y^*}\geq \real \scal{x^*-y^*,-y}\geq \real y^*(y)- |x^*(y)|>1-2\eps. \qedhere
\]
\end{proof}

\section{Super alternative Daugavet points and related notions}\label{section: Alternative Daugavet points and related notions}
We devote this section to give the main examples and properties of super AD points, obtaining also results for AD points. This section is divided into two subsections: the first one is devoted to providing examples and comparisions with other localization notions; the second subsection will deal with the implications of super AD points on the geometry of the underlying Banach space. 

Let us first provide a geometric characterization of the AD points and super AD points which we will use all along the section. The following characterizations of AD points follow directly from the proof of the global characterization of the ADP given in  \cite[Proposition~2.1]{MT04}. For every $x\in S_X$ and $\eps>0$, let \begin{equation*}
    \Delta_\eps(x):=\{y\in B_X\colon \norm{x-y}>2-\eps\}. 
\end{equation*}

\begin{proposition}\label{prop: Equivalent characterizations of AD points}
Let $X$ be a Banach space and let $x\in S_X$. Then, the following conditions are equivalent: 
\begingroup \parskip=0pt
\begin{enumerate}
    \item [\rm{(i)}] $x$ is an AD point; 
    \item[\rm{(ii)}] for every $x^*\in X^*$, the operator $T:=x^*\otimes x\colon X\to X$ given by $Tz=x^*(z)x$ for all $z\in X$ satisfies
        \[
        \max_{\theta\in\mathbb{T}}\norm{\Id+\theta T}= 1+\norm{T};
        \]
    \item [\rm{(iii)}] for every $f\in S_{X^*}$ and $\varepsilon>0$, there exists $y\in S_X$ such that $\abs{f(y)}>1-\varepsilon$ and $\norm{x+y}>2-\varepsilon$;
    \item[\rm{(iv)}] for every slice $S$ of $B_X$ and $\eps>0$, there exists $\theta\in\mathbb{T}$ and a subslice $T$ of $S$ such that for every $y\in T$, we have $\norm{x+\theta y}>2-\eps$;
    \item[\rm(v)] for every $\eps>0$, $B_X=\overline{\conv}\big(\mathbb{T}\Delta_{\varepsilon}(x)\big)$.
 \end{enumerate}
 \endgroup
\end{proposition}

For a super AD point $x$, condition (v) above is replaced with the weak denseness of the set $\mathbb{T}\Delta_{\varepsilon}(x)$ (instead of the denseness of its convex hull). Equivalently, we can provide a net characterization of super AD points which will be useful.

\begin{lemma}\label{lemma:SAD_net_characterization}
Let $X$ be a Banach space and let $x\in S_X$. The element $x$ is a super AD point if and only if for every $y\in S_X$, there exists $\theta\in \T$ and a net $(y_\alpha)$ in $B_X$ weakly convergent to $y$ such that $\norm{x+\theta y_\alpha}\to 2$. Moreover, if $\dim(X)=\infty$, then such a net can be found for every $y\in B_X$ and, additionally, can be taken in $S_X$. 
\end{lemma}

\subsection{Examples and comparison with other localization notions}\label{subsection:examplesandcomparision}
Our first goal is to discuss the relation between spears and super AD points. It is immediate that spear vectors are super AD points. The converse holds for spaces with the Kadec property (as it follows from the proof of Proposition~\ref{prop: K has the super ADP}).

\begin{proposition}\label{prop: Kadec property SAD implies spear}
Let $X$ be a Banach space and let $x\in S_X$. If $x\in \Spear(X)$, then $x$ is a super AD point. Moreover, if $X$ has the Kadec property (in particular, if $X$ is finite-dimensional or $X$ is asymptotically uniformly convex), then the set of super AD points coincides with $\Spear(X)$.
\end{proposition}

This result easily allows us to determine the super AD points in some concrete spaces, by using the description of the set of spear vectors given in \cite[Example 2.12]{KMMP}. Recall that the space $\ell_1(\Gamma)$ is asymptotically uniformly convex, and hence has the Kadec property for every nonempty set $\Gamma$.

\begin{example}\label{example:ellinftym-ell1Gamma} We give the description of the super AD points in some spaces:
\begingroup \parskip=0pt
\begin{enumerate}
  \item The super AD points of $\ell_\infty^m$ ($m\in \N$) are the elements whose coordinates have all modulus one.
  \item The super AD points of $\ell_1(\Gamma)$ (in particular, of $\ell_1^m$ ($m\in \N$) and of $\ell_1$) are the elements whose coordinates are zero except for one which has modulus one (that is, all  rotations of the elements of the canonical basis).
\end{enumerate} 
\endgroup 
\end{example}

The following example shows that the ADP does not necessarily imply the existence of any super AD point. 

\begin{example}\label{example: c_0 AD but no SAD points}
The space $c_0$ has the ADP \cite{MT04}, and hence all points on its unit sphere are AD points. However, there are no super AD points in $c_0$.  
\begin{proof}
Indeed, fix $x_0=(x_n)_{n\in\mathbb{N}} \in S_{c_0}$ and find an index $N\in\mathbb{N}$ such that
$\abs{x_n}<1/4$ whenever $n\geq N$. Consider the relatively weakly open set
\[
W:= \{ (y_n)\in B_{c_0}\colon \abs{y_n}<1/4, \; n=1,2,\dots,N\}
\]
and notice that for every $\theta \in\mathbb{T}$ and every element $y\in W$, we have
\begin{align*}
\norm{x_0+\theta y} = \sup_n \abs{x_n+\theta y_n} \leq \sup_n \abs{x_n}+\abs{y_n} < 1+ \frac{1}{4}< 2.
\end{align*}
(This is true, since if $n<N$, we have $\abs{y_n}<1/4$, and if $n\geq N$, then $\abs{x_n}< 1/4$.) Hence $x$ cannot be super AD.
\end{proof}
\end{example}

We next would like to clarify the relations between AD and super AD points and the various localizations of the DP defined previously in the literature. 

Clearly, every super Daugavet point is a super AD point, and every Daugavet point is an AD point. Let us note that the latter implication actually holds for $\nabla$-points.

\begin{proposition} 
Let $X$ be a Banach space. If $x\in S_X$ is a $\nabla$-point, then $x$ is an AD point.
\begin{proof}
    First, assume $X$ is a real Banach space. Consider a slice $S$ of $B_X$ and $\varepsilon>0$. We want to find $y\in S$ such that $\max\{\norm{x-y}, \norm{x+y}\}>2-\varepsilon$. There are two cases: if $x$ belongs to $S$, we can take $y:=x$. In the other case, we can find the desired $y\in S$ simply by the fact that $x$ is a $\nabla$-point. If $X$ is a complex Banach space, then every $\nabla$-point is already a Daugavet point \cite[Proposition~2.2]{LRT2024}, hence an AD point.
\end{proof}
\end{proposition}

An overview of the relations of both new and already known notions is given in Figure \ref{fig:relations between notions}. Examples in \cite{HLPV23,MPRZ} as well as the fact that the super ADP lies strictly between the ADP and the Daugavet property (Example~\ref{examples: ell1gamma not super ADP} and Theorem~\ref{theorem: renorming having SAD but not the DP}), show that none of the above implication reverses and that super AD points are not necessarily Daugavet points and vice versa. Besides, we will provide an example showing that super AD points are not necessarily $\nabla$-points (Example~\ref{example:superAD-no-nabla}).
\begin{figure}[h!]
     \centering
\begin{tikzpicture}[
    node distance=1.5cm and 2.5cm,
    every node/.style={draw, rectangle, minimum width=1.5cm, minimum height=0.6cm, align=center},
    smallbox/.style={draw, rectangle, minimum width=0.6cm, minimum height=0.6cm, align=center},
    implies/.style={draw, double, -implies, double distance=1.5pt, shorten >=2pt, shorten <=2pt}
]

\node (ccsDaugavet) {ccs Daugavet};
\node[below=1.5cm of ccsDaugavet] (superDaugavet) {super Daugavet};
\node[below=1.5cm of superDaugavet] (Daugavet) {Daugavet};

\node[right=2.5cm of superDaugavet] (superAD) {super AD};
\node[below=1.5cm of superAD, smallbox] (AD) {AD};

\node[below=1.5cm of Daugavet, xshift=2.4cm, smallbox] (nablaPoint) {\(\nabla\)};

\draw[implies] (ccsDaugavet.south) -- (superDaugavet.north);
\draw[implies] (superDaugavet.south) -- (Daugavet.north);

\draw[implies] (superAD.south) -- (AD.north);

\draw[implies] (superDaugavet.east) -- (superAD.west);
\draw[implies] (Daugavet.east) -- (AD.west);

\draw[implies] (Daugavet.south) -- (nablaPoint.north);
\draw[implies] (nablaPoint.north) -- (AD.south);
\end{tikzpicture}
    \caption{Relations between the notions}
    \label{fig:relations between notions}
\end{figure}
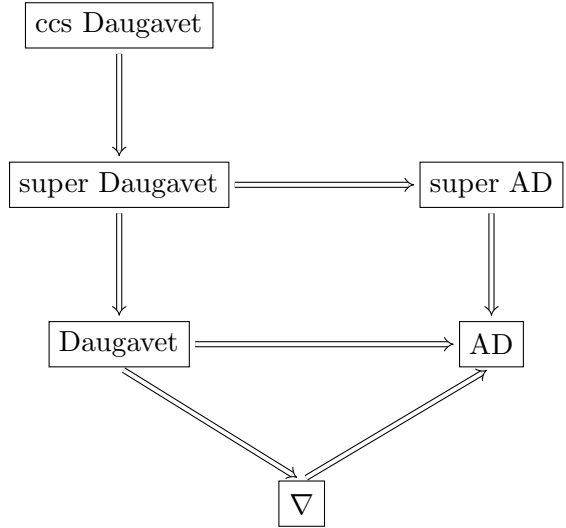

Our next aim is to illustrate the difference between super AD and super Daugavet points, showing that the later are super AD for all directions at the same time. 

\begin{proposition}\label{prop:superDpoint-is-superADPineverydirection}
Let $x\in S_X$ be a super Daugavet point. Then, for every $\eps>0$ and every non-empty relatively weakly open subset $W$ of $B_X$, there exists $y\in W$ such that $\norm{x+\theta y}>2-\eps$ for every $\theta\in\T$. 
\begin{proof}
From the definition of super Daugavet points, it is immediate to show that for every $\theta\in\T$, there exists $y_\theta\in W$ such that $\norm{x+\theta y_\theta}>2-\eps$ (it is enough to use that  $\overline{\theta} x$ is super Daugavet). Let $\{\theta_i\colon i=1,\ldots,n\}$ be an $\eps$-net in $\T$. Since the sets of the form 
\begin{equation*}
   \Delta_{\eps,\theta}(x):=\{y\in B_X\colon \norm{x+\theta y}>2-\eps\}
\end{equation*}
are relatively weakly open, by the argument above, we get that the set $W_1:=\Delta_{\eps,\theta_1}(x)\cap W$ is a non-empty relatively weakly open subset of $B_X$. Hence the set $W_2:=\Delta_{\eps,\theta_2}(x)\cap W_1$ is also a non-empty relatively weakly open subset, and iterating in this way, we get that the set $W\cap\bigcap_{i=1}^n\Delta_{\eps,\theta_i}(x)$ is non-empty. Now, every $y$ in this set satisfies $\norm{x+\theta_i y}>2-\eps$ for every $i\in\{1,\dots,n\}$, hence $\norm{x+\theta y}>2-2\eps$ for every $\theta\in \T$. The conclusion follows.   
\end{proof}
\end{proposition}

We can establish a similar statement for Daugavet points which shows the difference with AD points.

\begin{proposition}\label{prop: Daugavet point makes all rotations far for uniform element}
Let $X$ be a Banach space and let $x\in S_X$ be a Daugavet point. Then, for every slice $S\subset B_X$ and every $\varepsilon>0$, there exists a slice $T\subset S$ such that $\norm{x+\theta y}>2-\varepsilon$ for every $y\in T$ and every $\theta\in\mathbb{T}$.
\begin{proof}
Take a slice $S\subset B_X$, $\varepsilon>0$ and an $(\varepsilon/2)$-net $\{\theta_i\colon i=1,\ldots,n\}$ for $\mathbb{T}$. By \cite[Remark~2.3]{JungRueda} applied to $\overline{\theta_1} x$ which is a Daugavet point,  we can find a slice $S_1\subset S$ satisfying $\norm{x+\theta_1y}2-\varepsilon/2$ for any $y\in S_1$. Making use of \cite[Remark~2.3]{JungRueda} again, we can find for $\theta_2$ a slice $S_2\subset S_1\subset S$ so that $\norm{x+\theta_2y}>2-\varepsilon/2$ for any $y\in S_2$. Using this method $n$ times yields us a slice $S_n$ such that
\[
S_n\subset S_{n-1}\subset \dots \subset S_1\subset S
\]
and for any $y\in S_n$ we have $\norm{x+\theta_iy}>2-\varepsilon/2$. In conclusion, take $y\in T:=S_n$, $\theta\in\mathbb{T}$ and find $\theta_i\in\{1,\dots,n\}$ such that $\abs{\theta-\theta_i}<\varepsilon/2$. Then,
\[
\norm{x+\theta y}\geq \norm{x+\theta_i y} - \norm{\theta y-\theta_i y} >2-\frac{\varepsilon}{2}-\abs{\theta-\theta_i} > 2-\varepsilon.\qedhere
\]
\end{proof}
\end{proposition}

According to \cite[Proposition~3.1]{JungRueda} (respectively, \cite[Lemma~3.7]{MPRZ}), Daugavet points (respectively, super Daugavet points) are at distance $2$ from every denting point (respectively, point of continuity) of the unit ball. Actually, the same is true for $\nabla$-points \cite[Proposition 2.6]{HLPV23}. With AD points, we have a similar (weaker) result whose proof is straightforward.

\begin{lemma}\label{lemma: denting points far from AD points}
Let $X$ be a Banach space and let $x\in S_X$ be an AD point. Then, for each $y\in \dent(B_X)$ there exists $\theta\in\mathbb{T}$ so that $\norm{x+\theta y}=2$.
\end{lemma}

There is also a natural super AD counterpart of the previous lemma, also with straightforward proof.

\begin{lemma}\label{prop:SAD_distance_to_PC_points}
Let $X$ be a Banach space and let $x\in S_X$ be a super AD point. Then,  for every $z\in \PC(B_X)\cap S_X$, there exists $\theta\in\T$ such that $\norm{x+\theta z}=2$. 
\end{lemma}

Note that the two previous results are actually characterizations of AD points and super AD points in spaces with the RNP and the CPCP, respectively. In the first case, this is because the unit ball is the closed convex hull of its denting points; in the latter case, because the points of continuity are weakly dense in the unit ball. 

\begin{corollary}\label{corollary: characterizing AD and super AD points RNP and CPCP}
Let $X$ be a Banach space.
\begingroup \parskip=0pt
\begin{enumerate}
  \item If $X$ has the RNP, then $x\in S_X$ is an AD point if (and only if) for every $y\in \dent(B_X)$ there exists $\theta\in\mathbb{T}$ so that $\norm{x+\theta y}=2$.
  \item If $X$ has the CPCP, then $x\in S_X$ is a super AD point if (and only if) for every $z\in \PC(B_X)\cap S_X$, there exists $\theta\in\T$ such that $\norm{x+\theta z}=2$.
\end{enumerate}
\endgroup
\end{corollary}

We finish this subsection by proving that ``ccs AD points'' are actually ccs Daugavet points. Therefore, the only interesting stronger version of AD point is the one of super AD point.

\begin{proposition}\label{proposition: ccs AD is the same as ccs Daugavet}
Let $X$ be a Banach space and $x\in S_X$. Assume that for every convex combination of slices $C$ of $B_X$, we have that $\sup_{y\in C}\max_{\theta\in\T} \norm{x+\theta y}=2$. Then, $x$ is a ccs Daugavet point.
\end{proposition}

\begin{proof}
Fix $\varepsilon>0$ and a convex combination of slices $C:=\sum_{i=1}^n\lambda_i S_i$, where $S_i=S(B_X,f_i,\alpha_i)$, $\lambda_i>0$ for all $i=1,\ldots,n$, and $\sum_{i=1}^{n}\lambda_i=1$. It suffices to find $y\in C$ satisfying $\norm{x+y}>2-\varepsilon$. First, let us choose a $\varepsilon/2$-net $K:=\{\theta_1,\dots,\theta_m\}$ for $\mathbb{T}$. Consider the following convex combination of slices:
    \[
    D:= \frac{1}{m}\sum_{j=1}^m\sum_{i=1}^n\lambda_i S(B_X,\theta_jf_i,\alpha_i).
    \]
    By hypothesis, we can find an element $d\in D$ and $\theta_0\in\mathbb{T}$, for which $\norm{x+\theta_0d}>2-\varepsilon/4m$. See that
    \[
    d=\frac{1}{m}\sum_{j=1}^m\sum_{i=1}^n\lambda_i s^j_i,
    \]
    where $s^j_i\in S(B_X,\theta_jf_i,\alpha_i)$. Evidently, in this case, we have the inclusion $\overline{\theta_j}s^j_i\in S(B_X,f_i,\alpha_i)$. This implies that $\sum_i\lambda_i\overline{\theta_j}s^j_i\in C$ for each $j\in\{1,\dots,m\}$. Find $\theta_{N}\in K$ such that $\bigl|\theta_0-\overline{\theta_{N}}\bigr|<\varepsilon/2$. We now show that we can take $y:=\sum_i\lambda_i\overline{\theta_{N}}s^{N}_i\in C$, for which $\norm{x+y}>2-\varepsilon$. Assume, on the contrary, that $\norm{x+y}\leq 2-\varepsilon$. Then
    \begin{align*}
        2-\frac{\varepsilon}{4m}&< \norm{x+\theta_0 d}= \norm{x+\frac{\theta_0}{m}\sum_{j=1}^m\sum_{i=1}^n\lambda_i s^j_i}\\
        &\leq \frac{\norm{x+\theta_0\sum_{i=1}^n\lambda_is^{N}_i}}{m}+\sum_{j\neq N}\frac{\norm{x+\theta_0\sum_{i=1}^n\lambda_is^j_i}}{m}\\
        &\leq \frac{\norm{x+\theta_0\theta_N\sum_{i=1}^n\lambda_i\overline{\theta_N}s^{N}_i}}{m} +\frac{2(m-1)}{m}=\frac{\norm{x+\theta_0\theta_Ny}}{m}+\frac{2(m-1)}{m}.
    \end{align*}
    We can estimate the first term:
    \begin{align*}
\norm{x+\theta_0\theta_Ny}&\leq \norm{x+y}+\norm{\theta_0\theta_Ny-y}\\
    &\leq2-\varepsilon + \abs{\theta_0\theta_N-1}=2-\varepsilon+\bigl|\theta_0-\overline{\theta_{N}}\bigr| <2-\frac{\varepsilon}{2}.
    \end{align*}
Therefore, we conclude that
\begin{align*}
    2-\frac{\varepsilon}{4m}&< \frac{\norm{x+\theta_0\theta_Ny}}{m}+\frac{2(m-1)}{m}\\
    &<\frac{2-\varepsilon/2}{m}+\frac{2(m-1)}{m}=\frac{2m}{m}-\frac{\varepsilon}{2m}= 2-\frac{\varepsilon}{2m},
\end{align*}
which is a contradiction.
\end{proof}

We get the following characterization of the Daugavet property (compare it with the one in \cite{Kadets21}). 

\begin{corollary}\label{corollary: ccs-AD is DP}
Let $X$ be a Banach space. Suppose that for every $x\in S_X$ and every convex combination of slices $C$ of $B_X$ we have that $\sup_{y\in C}\max_{\theta\in\T} \norm{x+\theta y}=2$. Then, $X$ has the Daugavet property.
\end{corollary}

\begin{remark}
If $X$ has the strong diameter two property (that is, every convex combination of slices has diameter two), then every convex combination of slices almost reaches the unit sphere \cite[Theorem 3.1]{LMRZ2019}. Therefore, every spear in $X$ is a ``ccs AD point'', hence a ccs Daugavet point by Proposition~\ref{proposition: ccs AD is the same as ccs Daugavet}. 
\end{remark}

\subsection{Implications of AD points and super AD points on the geometry of the underlying space}\label{subsec:implications-of-points-in-geometry}
We first investigate the relations between rotundity properties of norms and AD points. To begin with, let us show that LUR points are never AD points.

\begin{proposition}\label{prop: LUR points no AD points}
Let $X$ be a Banach space with $\dim(X)>1$ and let $x\in S_X$ be a LUR point. Then, $x$ is not an AD point.
\end{proposition}

We need the following preliminary result.

\begin{lemma}\label{lemma: points far from AD point are around the opposite element}
Let $X$ be a Banach space with $\dim(X)>1$ and let $x\in S_X$. If there exists $\varepsilon,\delta>0$ such that $\Delta_{\varepsilon}(x)\subset B(-x,1-\delta)$, then $x$ is not an AD point.
\end{lemma}

\begin{proof}
To show that $x$ is not an AD point, it suffices to find $f\in S_{X^*}$ and $\eta,\xi>0$ such that for all $y\in S_X$, if $\abs{f(y)}> 1-\eta$, then $y\notin\Delta_\xi(x)$. Assume that we have $\varepsilon,\delta>0$ such that $\Delta_{\varepsilon}(x)\subset B(-x,1-\delta)$. Since $\dim(X)>1$, we can find $f\in S_{X^*}$ such that $f(x)=0$.  Take $y\in S_X$ such that $\abs{f(y)}>1-\delta$. Then 
\[
\norm{-x-y}\geq \abs{f(x+y)} = \abs{f(y)}>1-\delta.
\]
This means that $y\not\in B(-x,1-\delta)$, hence $y\not\in \Delta_{\varepsilon}(x)$, which concludes the proof.
\end{proof}

\begin{proof}[Proof of Proposition~\ref{prop: LUR points no AD points}]
Fix $\varepsilon\in(0,1)$. By definition of LUR, there exists $\delta>0$ such that every $y\in B_X$ satisfies the implication
    \[
    \norm{\frac{x+y}{2}}>1-\delta \implies \norm{x-y}\leq \varepsilon.
    \]
    Let us show that Lemma~\ref{lemma: points far from AD point are around the opposite element} holds for $x$. Take any $\varepsilon_1\in(0,2\delta)$ and $\delta_1\in(0,1-\varepsilon)$. Then we claim that $\Delta_{\varepsilon_1}(x)\subset B(-x,1-\delta_1)$. Indeed, if $y\in B_X$ is such that $\norm{x-y}>2-\varepsilon_1$, then
    \[
    \norm{\frac{x-y}{2}}>1-\frac{\varepsilon_1}{2} >1-\delta,
    \]
    and therefore
    \[
    \norm{-x-y}= \norm{x+y}\leq \varepsilon \leq 1-\delta_1.\qedhere
    \]
\end{proof}

As a consequence, we get the following result, related to \cite[Proposition~2.4]{KMMP2009}, where it is shown that spaces with the ADP do not contain LUR points.

\begin{corollary}\label{coro: UC spaces have no AD points}
Let $X$ be a Banach space with $\dim(X)>1$ which is LUR (in particular, if it is uniformly convex). Then, $X$ does not have any AD point.
\end{corollary}

We now deal with the relation between super AD points and differentiability of the norm. The following result can be immediately extracted from the proof of Proposition~\ref{prop:separable_CPCP_not_super_ADP}.

\begin{proposition}\label{prop:PC-wdense-GdifnosuperAD}
Let $X$ be an infinite-dimensional Banach space such that $\PC(B_X)$ is weakly dense in $B_X$ (for instance, if $X$ has the CPCP). If $x\in S_X$ is a point of G\^{a}teaux differentiability of the norm of $X$, then $x$ is not a super AD point. 
\end{proposition}

Observe that the analogous result for AD points is not true: $\ell_1$ has the RNP (hence, the CPCP), and has a dense subset of points of G\^{a}teaux differentiability of the norm (as being separable), and has the ADP, hence all elements in $S_{\ell_1}$ are AD points.

We will finally show that Example~\ref{example: c_0 AD but no SAD points} can be generalized to any asymptotically smooth spaces, for which we will introduce a bit of notation. Let $X$ be a Banach space. We denote by $\cof(X)$ the set of all subspaces of $X$ of finite codimension. Recall that the \emph{modulus of asymptotic smoothness} of $X$ at a point $x\in S_X$ is the function $\modaus{X}{x,\cdot}\colon \R^+\to \R^+$ given by 
\begin{equation*}
    \modaus{X}{x,t}:=\inf_{Y\in\cof(X)}\sup_{y\in S_Y}\{\norm{x+ty}-1\}.
\end{equation*}
It is easy to check that for a given $t\in\R^+$, $\modaus{X}{x,t}$ is the best constant $\rho$ for which we have the following: for every weakly null net $(x_\alpha)$ in $S_X$, $\limsup\norm{x+t x_\alpha}\leq 1+\rho$ (see e.g.\ \cite[Corollary~1.1.6]{PerreauPhDThesis}). The \emph{modulus of asymptotic smoothness} of the space $X$ is the function $\modaus{X}{\cdot}\colon \R^+\to \R_+$ given by \begin{equation*}
    \modaus{X}{t}:=\sup_{x\in S_X}\modaus{X}{x,t}.
\end{equation*} 
A point $x\in S_X$ is said to be \emph{asymptotically smooth} if $\lim_{t\to 0}\frac{\modaus{X}{x,t}}{t}=0$. The space $X$ is said to be \emph{asymptotically uniformly smooth} (\emph{AUS}, for short) if $\lim_{t\to 0}\frac{\modaus{X}{t}}{t}=0$. Prototypical examples of AUS spaces are $\ell_p$-sums ($p>1$) and $c_0$-sums of finite-dimensional spaces. We get that super AD point fails to be asymptotically smooth in the strongest possible way.

\begin{proposition}\label{prop: frechet diff points is not super AD point}
Let $X$ be an infinite-dimensional Banach space. If $x\in S_X$ is a super AD point, then $\modaus{X}{x,t}=t$ for every $t\in (0,1)$, hence $x$ is not asymptotically smooth (in particular, if $x$ is not a point of Fr\'{e}chet differentiability of the norm). 
\end{proposition}

\begin{proof}
Assume that $x$ is a super AD point. Then, by Lemma~\ref{lemma:SAD_net_characterization} there exists a net $(x_\alpha)$ in $S_X$ which converges weakly to $0$ and such that $\lim_\alpha\norm{x+x_\alpha}=2$. Then, by convexity, $\lim_\alpha\norm{x+tx_\alpha}=1+t$ for every $t\in (0,1)$, and this implies that $\modaus{X}{x,t}=t$. Hence, $x$ is not an asymptotically smooth point. 
\end{proof}

As a consequence, we extend the class of Banach spaces which fail to have super AD points.

\begin{corollary}
Let $(E_n)$ be a sequence of finite-dimensional spaces and let $p\in(1,\infty)$. Then, the spaces $\left[\bigoplus_{n\in\N}E_n\right]_{\ell_p}$ and  $\left[\bigoplus_{n\in\N}E_n\right]_{c_0}$ have no super AD points.
\end{corollary}

\section{Super AD points and the super ADP in some classical Banach spaces}\label{section: AD and super AD points in classical Banach spaces}
In this section we aim at characterizing super AD points and the super ADP in some Banach spaces, namely spaces $L_1(\mu)$ and $C(K)$, as well as their vector-valued versions $L_1(\mu,X)$ and $C(K,X)$. We will first deal with super AD points in subsections~\ref{subsec: SAD in absolute sums} (for $\ell_1$- and $\ell_\infty$-sums) and \ref{subsec: Super AD points in C(K) and L_1(mu)} (for $L_1(\mu,X)$ and $C(K,X)$), and then apply our results to the global property in subsection~\ref{subsec:exmples of super AD spaces}. 

\subsection{Super AD points in \texorpdfstring{$\ell_1$}{l1}- and \texorpdfstring{$\ell_\infty$}{linfty}-sums of Banach spaces}\label{subsec: SAD in absolute sums}
We start by gathering some results concerning super AD points in $\ell_1$- and $\ell_\infty$-sums of Banach spaces.

\begin{proposition}\label{prop: SAD from summands to 1-sum}
Let $X$ and $Y$ be Banach spaces, let $x\in S_X$ and $y\in S_Y$, and let $a,b>0$ with $a+b=1$. Then, we have the following implications: 
\begingroup \parskip=0pt
\begin{enumerate}
        \item If $x$ is a super AD point in $X$, then $(x,0)$ is a super AD point in $X\oplus_1 Y$.
        \item If $x$ is a super AD point in $X$ and $y$ is a super Daugavet point in $Y$, then $(ax,by)$ is a super AD point in $X\oplus_1 Y$.
    \end{enumerate} 
    \endgroup
\end{proposition}

\begin{proof}
The first implication is straightforward, so we will only prove the second one. Assume that $x$ is super AD and that $y$ is super Daugavet, and pick $(u,v)\in B_{X\oplus_1 Y}$. We first suppose that $u,v\neq 0$. (If either $u=0$ or $v=0$, the proof is straightforward). Since $x$ is super AD, there exists $\theta\in \T$ and a net $(u_\alpha)$ in $B_X$ such that $(u_\alpha)$ converges weakly to $\frac{u}{\norm{u}}$ and $\norm{x+\theta u_\alpha}\to 2$. Since $y$ is super Daugavet, we can find a net $(v_\alpha)$ in $B_Y$ which converges weakly to $\frac{v}{\norm{v}}$ and such that $\norm{y+\theta v_\alpha}\to 2$. Then the net $(\norm{u}u_\alpha,\norm{v}v_\alpha)$ lives in $B_{X\oplus_1 Y}$, converges weakly to $(u,v)$, and is such that \begin{equation*}
     \bigl\|(ax,by)+\theta(\norm{u}u_\alpha,\norm{v}v_\alpha)\bigr\| =\norm{ax+\theta\norm{u}u_\alpha}+\norm{by+\theta\norm{v}v_\alpha}\to a+\norm{u}+b+\norm{v}=2
 \end{equation*}
(where we have used Remark~\ref{remark:convexityargument}).
Hence $(ax,by)$ is a super AD point in $X\oplus_1 Y$ by Lemma~\ref{lemma:SAD_net_characterization}. 
\end{proof}

Let us give a comment on the limitations of the above proof.

\begin{remark}
Since $\ell_1\equiv \ell_1 \oplus_1 \ell_1$, Example \ref{example:ellinftym-ell1Gamma} shows that, in general, it is not enough to assume that $x$ and $y$ are super AD points in order to get that $(ax,by)$ is a super AD point in $X\oplus_1 Y$. A simpler example is given by $\ell_1^2\equiv \K\oplus_1 \K$. 

Note that the reason why the above argument doesn't work in this context is that applying the super AD net characterization to $\{x,u\}$ and $\{y,v\}$ in the above proof can provide two different modulus one scalars $\theta,\omega\in\T$ for which the weakly convergent nets are far from the super AD points in the respective unit balls.
\end{remark}

In the other direction, things behave as we expect them to.

\begin{proposition}\label{prop: SAD points l_1 sum from sum to summands}
Let $X$ and $Y$ be Banach spaces, let $x\in S_X$ and $y\in S_Y$, and $a,b>0$ with $a+b=1$. Then, we have the following implications:
\begingroup \parskip=0pt
\begin{enumerate}
      \item If $(x,0)$ is a super AD point, then $x$ is a super AD point.
      \item If $(ax,by)$ is a super AD point, then $x$ and $y$ are super AD points. 
  \end{enumerate}  
  \endgroup
\end{proposition}

\begin{proof}
Assume that $x$ is not super AD. Then there exists $u\in S_X$ such that for every net $(u_\alpha)$ in $B_X$ which converges weakly to $u$ and for every $\theta\in\T$, we have $\limsup\norm{x+\theta u_\alpha}<2$. Let $(u_\alpha,v_\alpha)$ be a net in $B_{X\oplus_1 Y}$ which converges weakly to $(u,0)$. Then, $(u_\alpha)$ converges weakly to $u$ in $B_X$. Therefore, by the weak lower semi-continuity of the norm, we have that $\liminf\norm{u_\alpha}\geq \norm{u}=1$. Hence $\norm{v_\alpha}\to 0$. Furthermore, by our assumption, we have that for every $\theta\in\T$, $\limsup\norm{x+\theta u_\alpha}<2$. So fix $\theta\in\T$. On one hand, we have \begin{equation*}
        \limsup \norm{(x,0)+\theta(u_\alpha,v_\alpha)}=\limsup (\norm{x+\theta u_\alpha}+\norm{v_\alpha})\leq \limsup\norm{x+\theta u_\alpha}+\limsup\norm{v_\alpha}<2,
    \end{equation*} 
    so $(x,0)$ is not a super AD point.
On the other hand, we have \begin{align*}
        \limsup \norm{(ax,by)+\theta(u_\alpha,v_\alpha)} &=\limsup (\norm{ax+\theta u_\alpha}+\norm{by+\theta v_\alpha}) \\
        &\leq (1-a)\limsup\norm{u_\alpha}+a\limsup\norm{x+\theta u_\alpha}+b+\limsup\norm{v_\alpha} \\
        &<(1-a)+2a+b=2(a+b)=2,
    \end{align*} so $(ax,by)$ is not a super AD point. 
\end{proof}

We now deal with super AD points in $\ell_\infty$-sums.

\begin{proposition}\label{prop:infty sum super AD+super AD gives super AD}
Let $X$ and $Y$ be Banach spaces, and let $x\in S_X$ and $y\in S_Y$. If $x$ and $y$ are super AD points, then $(x,y)$ is a super AD point in $X\oplus_{\infty}Y$.
\end{proposition}

\begin{proof}
Let $(u,v)\in S_{X\oplus_\infty Y}$. Then $\norm{u}=1$ or $\norm{v}=1$. In the first case, since $x$ is super AD, there exists a net $(u_\alpha)$ in $B_X$ and $\theta\in\T$ such that $(u_\alpha)$ converges weakly to $u$ and $\norm{u+\theta u_\alpha}\to 2$ by Lemma~\ref{lemma:SAD_net_characterization}. Then $(u_\alpha,v)$ converges weakly to $(u,v)$ in $B_{X\oplus_\infty Y}$ and $\norm{(u,v)+\theta(u_\alpha,v)}\geq \norm{u+\theta u_\alpha}\to 2$. The other case is analogous, using that $y$ is super AD. 
\end{proof}

When one of the spaces is infinite-dimensional, being super AD point in the $\ell_\infty$-sum is easier.

\begin{proposition}\label{proposition: infty sums X infinite dim x super AD means (x,y) super AD}
Let $X$ be an infinite-dimensional Banach space and let $Y$ be an arbitrary Banach space. If $x$ is a super AD point in $X$, then $(x,y)$ is a super AD point in $X\oplus_{\infty} Y$ for every $y\in B_Y$.
\end{proposition}

\begin{proof}
The proof is analogous to the above one, but it is not required here to work by cases, because the assumption $\dim(X)=\infty$ allows to produce the desired net in $B_X$ for every $u\in B_X$ by the moreover part of Lemma~\ref{lemma:SAD_net_characterization}, and not merely for unit sphere elements. 
\end{proof}

For the converse result, we have the following two versions. The first one deals with the case when one of the factors is finite-dimensional.

\begin{proposition}\label{prop: infty sum (x,y) super AD and Y fin-dim, then x is super AD}
Let $X$ and $Y$ be Banach spaces. If $(x,y)\in S_{X\oplus_\infty Y}$ is a super AD point in $X\oplus_{\infty}Y$ and $Y$ is finite-dimensional, then $x$ is a super AD point (and, in particular, $\|x\|=1$).
\end{proposition}

\begin{proof}
    Assume that $x$ is not a super AD point. Then there exits $u\in S_X$ such that for every net $(u_\alpha)$ in $B_X$ which converges weakly to $u$ and for every $\theta\in\T$, we have $\limsup_\alpha\norm{x+\theta u_\alpha}<2$. Pick a net $(u_\alpha,v_\alpha)$ in $B_{X\oplus_\infty Y}$ which converges weakly to $(u,0)$. Then $(u_\alpha)$ converges weakly to $u$, and $(v_\alpha)$ converges weakly to $0$. Since $\dim(Y)<\infty$, it follows that $\norm{v_\alpha}\to 0$. Furthermore, from the above, we have $\limsup_\alpha\norm{x+\theta u_\alpha}<2$ for every $\theta\in \T$. Hence 
    \begin{align*}
        \limsup_\alpha\norm{(x,y)+\theta(u_\alpha,v_\alpha)} & =\limsup_\alpha\max\{\norm{x+\theta u_\alpha},\norm{y+\theta v_\alpha}\} \\ &\leq \max\left\{\limsup_\alpha\norm{x+\theta u_\alpha},\norm{y}\right\}<2.
    \end{align*}
     Therefore, $(x,y)$ is not a super AD point in $X\oplus_\infty Y$.
\end{proof}

Note that as a consequence, we can obtain the description of the super AD points in $\ell_\infty^m$ which we already obtained in Example~\ref{example:ellinftym-ell1Gamma} using spear vectors: the set of super AD points in $\ell_\infty^m$ is equal to the set $\T^m$ for every $m\in \N$.

In the general case for the converse result, we obtain that if a pair is a super AD point of an $\ell_\infty$-sum, then at least one of the coordinate has to be a super AD point. The proof is analogous to the one given for Proposition~\ref{prop: infty sum (x,y) super AD and Y fin-dim, then x is super AD}.

\begin{proposition}\label{proposition: infty sum (x,y) is super AD, then either x or y is super AD}
Let $X$ and $Y$ be Banach spaces. If $(x,y)\in S_{X\oplus_\infty Y}$ is a super AD point in $X\oplus_{\infty}Y$, then $x$ is a super AD point in $X$ or $y$ is a super AD point in $Y$.
\end{proposition}

\subsection{Super AD points in spaces of integrable and of continuous functions}\label{subsec: Super AD points in C(K) and L_1(mu)}
As previously mentioned, in the classical function spaces $L_1(\mu)$ and $C(K)$, every point in the unit sphere is an AD point (since these spaces have the ADP), but these points do not always satisfy any stronger notion. In this subsection, we study in detail super AD points in these spaces and in their vector valued counterparts.

We start with the case of spaces of (vector-valued) integrable functions. Let $(S,\Sigma,\mu)$ be a measure space. Recall that a measurable set $A\subset S$ is called an \emph{atom} for $\mu$ if $\mu(A)>0$ and if $\mu(B)=0$ for every measurable subset $B\subset A$ such that $\mu(B)<\mu(A)$. It was proved in \cite[Theorem~4.8]{MPRZ} that Daugavet and super Daugavet points coincide in the space $L_1(\mu)$, and that these points are precisely the norm one functions in $L_1(\mu)$ whose support contain no atom. As we have already seen, the situation is different for super AD points, as $\ell_1$ contains points which are super AD but does not contain any  Daugavet point. Let $X$ be a Banach space. Recall that a measurable function $f$ from $S$ to $X$ is almost everywhere constant on every atom. We denote by $L_1(\mu,X)$  the Banach space of all $X$-valued $\mu$-integrable functions. We can decompose $L_1(\mu,X)$ as 
\begin{equation}\label{eq:Maharam}
    L_1(\mu,X)=L_1(\nu,X)\oplus_1 \ell_1(\Gamma,X),
\end{equation} 
where $\nu$ is the continuous part of $\mu$, $\Gamma$ is the set of all atoms for $\mu$ (up to a measure 0 set), and $\ell_1(\Gamma,X)=\left[\bigoplus_{\gamma\in\Gamma}X\right]_{\ell_1}$ stands for the $\ell_1$-sum of $\Gamma$ copies of the space $X$, see \cite[Theorem~2.1]{Johnson} for the scalar-valued case, the vector-valued case is a consequence of the latter result. The following result is an immediate corollary of \cite[Theorem~4.11]{MPRZ} together with the stability results under $\ell_1$-sums for Daugavet and super Daugavet points (see \cite[Section~4]{AHLP} and \cite[Section~3.2]{MPRZ}). 

\begin{proposition}\label{prop:Daugavet points in L_1(mu,X)}
Let $(S,\Sigma,\mu)$ be a measure space, let $X$ be a Banach space, and let $f\in S_{L_1(\mu,X)}$. Then, $f$ is a Daugavet (respectively, super Daugavet) point if and only if $f(s)=0$ or $\frac{f(s)}{\norm{f(s)}}$ is a Daugavet (respectively, super Daugavet) point almost everywhere on every atom.
\end{proposition}

From the results from the previous subsection, we deduce in a similar way the following characterization for super AD points. 

\begin{proposition}\label{prop:super AD points in L_1(mu,X)}
Let $(S,\Sigma,\mu)$ be a measure space, let $X$ be a Banach space, and let $f\in S_{L_1(\mu,X)}$. Then, $f$ is a super AD point if and only if $f(s)=0$ or $\frac{f(s)}{\norm{f(s)}}$ is a super Daugavet point almost everywhere on all but one atom, where $\frac{f(s)}{\norm{f(s)}}$ is a super AD point.
\end{proposition}

\begin{proof}
Observe that the function $f$ can be seen as the element $(f_c,(f_\gamma)_{\gamma\in \Gamma})$ in the decomposition \eqref{eq:Maharam}, where $f_c$ corresponds to the restriction of the function $f$ to the continuous part of $S$ and $f_\gamma$ corresponds to the a.e.\ value of $f$ on the atom $\gamma$ for every $\gamma\in\Gamma$. Note that since $\nu$ is atomless, we have that either $f_c=0$ or $\frac{f_c}{\norm{f_c}}$ is a super Daugavet point in $L_1(\nu,X)$. Therefore, using Propositions~\ref{prop: SAD from summands to 1-sum} and \ref{prop: SAD points l_1 sum from sum to summands}, we immediately infer that $f$ is a super AD point if and only $f_\gamma=0$ or $\frac{f_\gamma}{\norm{f_\gamma}}$ is a super Daugavet point for all $\gamma\in\Gamma$ except one $\gamma_0\in \Gamma$ for which $\frac{f_{\gamma_0}}{\norm{f_{\gamma_0}}}$ is a super AD point. 
\end{proof}

In the scalar-valued case, as $X=\K$ contains no (super) Daugavet points, we get the following result.

\begin{corollary}
Let $(S,\Sigma,\mu)$ be a measure space and let $f\in S_{L_1(\mu)}$. Then, $f$ is a super AD point if and only if $f$ vanishes almost everywhere on all but one atom.
\end{corollary}

Note that this later result also allows to give a concrete example of a super AD point which is not a $\nabla$-point. 

\begin{example}\label{example:superAD-no-nabla}
Let $(S,\Sigma,\mu)$ be a measure space which has non-trivial continuous and atomic parts, then the space $L_1(\mu)$ contain super AD points which are not $\nabla$-points. Indeed, take any norm one function $f$ whose support contains a non-trivial continuous part and a single atom. On the one hand, $f$ is a super AD point by the above corollary. On the other hand, in the real case, by \cite[Proposition 3.7]{HLPV23} $\nabla$-points in $L_1(\mu)$ are either Daugavet points (i.e.\ elements whose support contains no atoms \cite[Theorem~4.8]{MPRZ}),  or elements of the form  $f =\theta\mu(A)^{-1}\chi_A$ with $\theta\in\{-1,1\}$ and $A$ is an atom for $\mu$. In the complex case, $\nabla$-points are Daugavet points \cite[Proposition 2.2]{LRT2024}. 
\end{example}

Now we move to the context of $C(K)$-spaces. Let $K$ be a infinite compact Hausdorff space. From \cite[Corollary~4.3]{MPRZ}, \cite[Theorem 4.2]{MPRZ}, and \cite[Theorem 3.4]{AHLP} we get that the notions of Daugavet, super Daugavet and ccs Daugavet points are equivalent in $C(K)$ and that a function $f$ belong to this class if and only if it attains its norm at a cluster point of $K$. We first prove that, in this setting, super AD points also fall in the same category. 

\begin{proposition}\label{prop: super AD points coincide with other notions in C(K)}
Let $K$ be a infinite compact Hausdorff space. If $f\in S_{C(K)}$ does not attain its norm at a cluster point of $K$, then $f$ is not a super AD point.
\begin{proof}
Let $H:= \{t \in K\colon \abs{f(t)}=1\}$. By assumption, every $t\in H$ is an isolated point of $K$, so we can find $\delta\in(0,1)$ such that $\abs{f(t)}\leq 1-\delta$ whenever $t\in K\setminus H$ (otherwise we could construct some $t\in H$ as a cluster point of elements from $K\setminus H$). Furthermore, the singleton $\{t\}$ is an open subset of $K$ for every $t\in H$ so, by compactness, it follows that the set $H$ is finite. Hence, the set
 \begin{equation*}
       W:= \{ g\in B_{C(K)} \colon \abs{g(t)}<\delta, \; t\in H\}
\end{equation*}
is a relatively weakly open subset of $B_{C(K)}$. Take any $g\in W$ and $\theta \in \mathbb{T}$. Notice that, if $t\in H$, we have
\begin{equation*}
    \abs{f(t)+\theta g(t)}\leq \abs{f(t)}+\abs{g(t)}< 1 +\delta <2.
\end{equation*} On the other hand, if $t\in K\setminus{H}$:
  \begin{equation*}
    \abs{f(t)+\theta g(t)}\leq \abs{f(t)}+\abs{g(t)} < 1-\delta +1 = 2-\delta < 2.
   \end{equation*}
    Therefore $\norm{f+\theta g} < 2$, and $f$ is not a super AD point. 
    \end{proof}
\end{proposition}

Let $K$ be a infinite compact Hausdorff space and let $X$ be an infinite-dimensional Banach space, and let $C(K,X)$ be the space of all continuous $X$ valued functions on $K$ endowed with the supremum norm. Let $f\in S_{C(K,X)}$. By \cite[Theorem 4.2]{MPRZ}, if the norm of $f$ is attained at a cluster point of $K$, then $f$ is a ccs Daugavet point so, in particular, a super AD point. Recall that if $t_0\in K$ is an isolated point of $K$, we obtain the decomposition $C(K,X) = C(K\setminus \{t_0\}, X)\oplus_{\infty} X$, where the isometry is given by the mapping $f \mapsto (f|_{K\setminus\{t_0\}},f(t_0))$. For functions attaining their norm only at isolated points of $K$, we have the following characterization of super AD points.

\begin{proposition}\label{prop: super AD in C(K,X)}
Let $K$ be a infinite compact Hausdorff space and let $X$ be an infinite-dimensional Banach space.  Assume that $f\in S_{C(K,X)}$ does not attain its norm at a cluster point of $K$. Then, $f$ is a super AD point if and only if $f(t_0)$ is a super AD point for some isolated point $t_0\in K$.
\end{proposition}

\begin{proof}
 Observe that for every isolated point $t_0$ in $K$, we can identify the function $f$ with the element $(f|_{K\setminus\{t_0\}}, f(t_0))$ in the above decomposition. So from Proposition~\ref{proposition: infty sums X infinite dim x super AD means (x,y) super AD}, we immediately get that if $f(t_0)$ is a super AD point, then $f$ is also a super AD point.
  For the other direction note that, as in the previous proof, we can find $\delta\in(0,1)$ such that $\norm{f(t)}\leq 1-\delta$ whenever $t\in K\setminus H$, and we have that the set $H:=\{t\in K\colon \norm{f(t)}=1\}$ is finite. Iterating the above decomposition, we can write 
    \begin{equation*}
        C(K,X) = (C(K\setminus H, X)\oplus_{\infty} Y ,
    \end{equation*}
    where $Y = \left[\bigoplus_{t\in H} X\right]_{\ell_{\infty}}$. In particular, $f$ can be identified with the element $(f|_{K\setminus H},(f(t))_{t\in H})$ in this space. Since, $\norm{f|_{K\setminus H}}<1$, it follows from Proposition \ref{proposition: infty sum (x,y) is super AD, then either x or y is super AD} that $(f(t))_{t\in H}$ is a super AD point in $Y$. Since $H$ is finite, we get from the same proposition that $f(t_0)$ must be a super AD point for some $t_0\in H$.
\end{proof}

\subsection{Characterizing the super ADP in spaces of integrable and of continuous functions}\label{subsec:exmples of super AD spaces}
Our final goal in the paper is to apply our previous results about super AD points to characterize the super ADP for $\ell_1$- and $\ell_\infty$-sum of spaces and for vector-valued continuous or integrable function spaces. 

First, combining Propositions~\ref{prop: SAD from summands to 1-sum} and \ref{prop: SAD points l_1 sum from sum to summands}, we get the following.

\begin{proposition}\label{proposition: super ADP in ell1sums}
Let $X$ and $Y$ be Banach spaces. The space $X\oplus_1 Y$ has the super ADP if and only if $X$ has the DP and $Y$ has the super ADP (or vice versa).
\end{proposition}

In particular, let us note that this result also provides a simple way of renorming spaces with the Daugavet property to obtain the super ADP without the DP which is arguably simpler than the one provided in Theorem~\ref{theorem: renorming having SAD but not the DP}.

\begin{example}
Let $X$ be a Banach space with the Daugavet property. Take a one-codimensional subspace $Y$ of $X$, which also has the Daugavet property \cite[Theorem~2.14]{kssw}. Then, the space $Y\oplus_1 \K$ is isomorphic to $X$, has the super ADP by Proposition~\ref{proposition: super ADP in ell1sums}, but it does not have the Daugavet property since $(0,1)$ is clearly a denting point of the unit ball of $Y\oplus_1 \K$. 
\end{example}

For $L_1(\mu,X)$ spaces, the following result is obtained by combining Propositions~\ref{prop:Daugavet points in L_1(mu,X)} and \ref{prop:super AD points in L_1(mu,X)}.

\begin{theorem}\label{theorem:characterizing-superADPL1muX}
Let $(S,\Sigma,\mu)$ be a measure space and let $X$ be a Banach space.    The space $L_1(\mu,X)$ has the super ADP if and only if one of the following three conditions is satisfied.  
\begingroup \parskip=0pt
\begin{enumerate}
        \item $\mu$ is atomless;
        \item $X$ has the DP;
        \item $\mu$ has exactly one atom and $X$ has the super ADP. 
    \end{enumerate} 
    \endgroup
    In particular, in order for $L_1(\mu,X)$ to have the super ADP but not the DP, then $\mu$ must have exactly one atom and $X$ itself must have the super ADP and fail the DP.
\end{theorem}

Contrary to the case of $\ell_1$-sums, from Proposition~\ref{prop: infty sum (x,y) super AD and Y fin-dim, then x is super AD} we get that in order for an $\ell_\infty$-sum of Banach spaces to have the super ADP, then both summands have to be infinite-dimensional. Combining Propositions~\ref{proposition: infty sums X infinite dim x super AD means (x,y) super AD} and \ref{proposition: infty sum (x,y) is super AD, then either x or y is super AD}, we then get the following. 

\begin{proposition}\label{proposition:ellinfty sum super ADP}
Let $X$ and $Y$ be Banach spaces. The space $X\oplus_\infty Y$ has the super ADP if and only if $X$ and $Y$ are both infinite-dimensional spaces, and $X$ has the DP and $Y$ has the super ADP (or vice versa).
\end{proposition}

We may get from this result an example showing that, contrary to the case of the Daugavet property and the ADP, the super ADP does not always pass from a dual Banach space to its predual.

\begin{example}
Let $X=L_1[0,1]\oplus_\infty\K$. It follows from Proposition~\ref{proposition:ellinfty sum super ADP} that $X$ fails the super ADP. However, $X^*=L_\infty[0,1]\oplus_1\K$ does have the super ADP by Proposition~\ref{proposition: super ADP in ell1sums}. 
\end{example}

Finally, we get the following characterization of the super ADP for $C(K,X)$ spaces combining the classical result about the Daugavet property in $C(K,X)$ spaces for perfect $K$ with Proposition~\ref{prop: super AD in C(K,X)}.

\begin{theorem}\label{theorem:characterizing-superADPC(K,X)}
Let $K$ be a compact Hausdorff space and $X$ be a Banach space. Then, we have the following statements. 
\begingroup \parskip=0pt
\begin{enumerate}
        \item If $K$ is perfect, then $C(K,X)$ has the Daugavet property.  
        \item If $K$ is not perfect, then $C(K,X)$ has the super ADP if and only if $X$ is a infinite-dimensional space with the super ADP or $K$ is a singleton and $X:=\K$ (i.e.\ $C(K,X)\equiv \K$).
    \end{enumerate}
\endgroup
\end{theorem}

\section{Open questions}\label{section: open questions}
We end the paper with some related questions.

In Theorem~\ref{theorem: super ADP is not Asplund}, we proved that if an infinite-dimensional Banach space has the super ADP, then the norm is $1$-rough. All the examples of super ADP spaces we know are actually $2$-rough, hence we wonder if this is so in general.
\begin{question}
    Let $X$ be an infinite-dimensional Banach space with the super ADP. Is it true that the norm of $X$ is $2$-rough? Or equivalently, is it true that the space $X$ is octahedral?
\end{question}

By the geometric characterization of the Daugavet property in terms of convex combinations of slices, it follows that there is no Daugavet space which is strongly regular. In Theorem~\ref{theorem: super ADP fails CPCP}, we showed that there is no infinite-dimensional CPCP space with the super ADP property. This leaves open whether there could be a super ADP space, which is strongly regular. 

\begin{question}
    Is there an infinite-dimensional super ADP space, which is strongly regular?
\end{question}

Slicely countably determined Banach spaces (SCD for short) were introduced in \cite{AKMMS09} as a natural joint generalization of separable strongly regular (and in particular RNP) Banach spaces and separable Banach spaces spaces not containing $\ell_1$ (and in particular separable Asplund spaces). It was proved there that separable Banach spaces with the Daugavet property fails to be SCD. So it is natural to ask the following. 

\begin{question}
    Does there exist an infinite-dimensional SCD space with the super ADP?
\end{question}

As mentioned in the introduction, spaces with the Daugavet property do not embed into spaces with an unconditional basis. We do not know is this is the case for spaces with the super ADP.

\begin{question}
    Does there exist an infinite-dimensional space with an unconditional basis (or even with a one-unconditional basis) and the super ADP?
\end{question}

Two more questions which remain open to the best of our knowledge are the following.

\begin{question}
    Is there an infinite-dimensional space with the super ADP but without Daugavet points?
\end{question}

\begin{question}
    Let $X$ be an infinite-dimensional dimensional super ADP space. Does $X$ contain a copy of $\ell_1$? Does $X^*$ fail the CPCP?
\end{question}

\section*{Acknowledgements}

The work of J.~Langemets, M.~L\~{o}o, and Y.~Perreau was supported by the Estonian Research Council grant (PRG2545). M.\ Mart\'in and A.\ Rueda Zoca were partially supported by the grant PID2021-122126NB-C31 funded by MICIU/AEI/10.13039/501100011033 and ERDF/EU, by ``Maria de Maeztu'' Excellence Unit IMAG, funded by MICIU/AEI/10.13039/501100011033 with reference CEX2020-001105-M and by Junta de Andaluc\'ia: grant FQM-0185.  The research of A.\ Rueda Zoca was also supported by Fundaci\'on S\'eneca: ACyT Regi\'on de Murcia: grant 21955/PI/22.

\bibliographystyle{amsalpha}

\end{document}